\documentclass[11pt,preprint]{imsart}
\usepackage{amsthm,amsmath,amssymb}
\usepackage{array} 
\usepackage{paralist} 
\usepackage{cite}
\usepackage[round]{natbib}
\usepackage{color}
\allowdisplaybreaks

\arxiv{1804.04108}

\startlocaldefs

\theoremstyle{plain}
\newtheorem{theorem}{Theorem}[section]
\newtheorem{lemma}[theorem]{Lemma}
\newtheorem{proposition}[theorem]{Proposition}

\theoremstyle{definition}
\newtheorem{definition}[theorem]{Definition}
\newtheorem{assumption}[theorem]{Assumption}
\newtheorem*{notation}{Notation}

\theoremstyle{remark}
\newtheorem{remark}[theorem]{Remark}

\endlocaldefs

\begin{document}

\begin{frontmatter}

\title{LAN property for stochastic differential equations driven by fractional Brownian motion of Hurst parameter $H\in(1/4,1/2)$}
\runtitle{LAN for SDE driven by fBM of Hurst parameter $H\in(1/4,1/2)$}

\author{\fnms{Kohei} \snm{Chiba}\ead[label=e1]{kchiba@ms.u-tokyo.ac.jp}\thanksref{t1}}
\thankstext{t1}{This work was supported by a Grant-in-Aid for JSPS Fellows 18J20239, JST CREST and the Program for Leading Graduate Schools, MEXT, Japan.
} 
\address{3-8-1 Komaba, Meguro, Tokyo 153-8914 Japan. \printead{e1}}
\affiliation{Graduate School of Mathematical Sciences, the University of Tokyo}


\runauthor{Kohei Chiba}

\begin{abstract}
In this paper, we consider the problem of estimating the drift parameter of solution to the stochastic differential equation driven by a fractional Brownian motion with Hurst parameter less than $1/2$ under complete observation. We derive a formula for the likelihood ratio and prove local asymptotic normality when $H \in (1/4,1/2)$. Our result shows that the convergence rate is $T^{-1/2}$ for the parameters satisfying a certain equation and $T^{-(1-H)}$ for the others. 
\end{abstract}



\end{frontmatter}


%
%
%
%
%
\makeatletter
\@addtoreset{equation}{section}
\def\theequation{\thesection.\arabic{equation}}
\makeatother
%
%
%
%
%


\tableofcontents

\section{Introduction}

Let $(X_{t}^{\theta})_{t\in[0,T]}$ be a solution of the (one-dimensional) stochastic differential equation
\begin{align}\label{eq1}
dX_{t} = a(X_{t},\theta)dt + \sigma dB_{t},\ X_{0} = x_{0},\ t\in[0,T].
\end{align}
Here, $X_{0}=x_{0} \in \mathbb{R}$ is an initial condition, $\{a(\cdot,\theta^{\prime})\mid\theta^{\prime}\in\Theta\}$ is a parametrized family of drift coefficients, $\Theta$ is an open subset of $\mathbb{R}^{m}$, $\sigma\neq0$ is a diffusion coefficient and $ B = (B_{t})_{t\in[0,T]} $ is a fractional Brownian motion with Hurst parameter $H\in(0,1)$ (later we restrict ourselves to the case where $ H\in(1/4,1/2)$).

We are interested in estimating $\theta \in \Theta$ from the completely observed data $(X_{t}^{\theta})_{t\in[0,T]}$ when $T\to\infty$. It is natural to consider the maximum likelihood estimation, which is successful in the case of ergodic diffusion processes (see \cite{kutoyants2004statistical}). \cite{kleptsyna2002statistical} proved strong consistency of the maximum likelihood estimator and derived explicit formulas for the asymptotic bias and mean square error when the observed process is the \textit{fractional Ornstein-Uhlenbeck process} with Hurst parameter $H\in[1/2,1)$. In the fractional Ornstein-Uhlenbeck case, \cite{brouste2010asymptotic} and \cite{bercu2011sharp} proved asymptotic normality of the maximum likelihood estimator in a different way. In the case where the drift coefficient $a$ is of the form $ a(x,\theta) = \theta a_{0}(x) $ for some function $a_{0}$, \cite{tudor2007statistical} proved strong consistency of the maximum likelihood estimator.

These results used an explicit expression for the maximum likelihood estimator. However, it is not available for a general drift coefficient $a$. As is done in \cite{kutoyants2004statistical}, we rely on Ibragimov and Has'minski\u{\i}'s framework when $\theta \mapsto a(\cdot,\theta)$ is nonlinear. That is, we can derive asymptotic properties of maximum likelihood estimator (and Bayes estimator) from weak convergence of likelihood ratio fields. 
Therefore, it is important to specify the weak limit of likelihood ratio fields. In many cases, the limit of the likelihood ratio fields $(Z_{\epsilon,\theta})_{\epsilon>0}$ is of the form 
\begin{align*}
Z_{0,\theta}(u) = \exp\left( u^{\star}\Delta(\theta) - \frac{1}{2}u^{\star}I(\theta)u \right),
\end{align*}
where $u\in\mathbb{R}^{m}$ ($u^{\star}$ denotes the transpose of $u$), $I(\theta)$ is an $m\times m$-positive definite symmetric matrix, and $\Delta(\theta)$ is an $m$-dimensional Gaussian random variable with mean zero and variance $I(\theta)$. \textit{Local asymptotic normality} ensures \textit{finite-dimensional convergence} of $(Z_{\epsilon,\theta})_{\epsilon>0}$ to $Z_{0,\theta}$, that is, $Z_{\epsilon,\theta}(u) \to^{d} Z_{0,\theta}(u)$ as $\epsilon\to 0$ for each $u\in\mathbb{R}^{m}$. This is the reason to investigate local asymptotic normality in this paper. 
Note that, to bridge local asymptotic normality and convergence of estimators, it is necessary to investigate further the properties of the likelihood ratio field (e.g., modulus of continuity and large deviation inequality).
This will be investigated in a future work. For details of Ibragimov and Has'minski\u{\i}'s theory, we refer to \cite{ibragimov1981statistical} and \cite{yoshida2011polynomial}.

Although we are motivated by the maximum likelihood estimator in this paper, there are many papers investigating other estimators. We mention some literature. Properties of the least-square type estimators are investigated in, for example, \cite{hu2010parameter, hu2017parameter, hu2018drift, neuenkirch2014least}. For the problem of estimating the drift parameter of fractional diffusion processes, we also refer to monographs \cite{mishura2008stochastic, kubilius2017parameter, rao2011statistical}.

Recently \cite{liu2015lan} proved local asymptotic normality in the case of $H\in(1/2,1)$. However, it is still unknown whether local asymptotic normality holds or not when $H$ is less than $1/2$. We partially solve this problem. More precisely, let $(\mu_{\theta}^{T})_{\theta\in\Theta}$ be the probability measures on the space of continuous functions induced by the solution of the equation (\ref{eq1}). The main aim of this paper is to prove local asymptotic normality of the probability measures $(\mu_{\theta}^{T})_{\theta\in\Theta}$ when $H\in (1/4,1/2)$. In order to do this, it is necessary to derive a likelihood ratio formula for the probability measures $(\mu_{\theta}^{T})_{\theta\in\Theta}$. We will derive a likelihood ratio formula in Section 4. Our proof of local asymptotic normality, which is given in Section 5, relies on properties of the stationary solution of the equation (\ref{eq1}). We will investigate some properties of the stationary solution of the equation (\ref{eq1}) in Section 3. The results are summarized in Section 2.

\section{Main results}

First we recall a definition of \textit{fractional Brownian motion} (abbreviated \textit{fBM}).

\begin{definition}
Let $(\Omega,\mathcal{F},\mathbb{P})$ be a probability space and $I$ be an interval $[0,T]$ $(T>0)$ (resp. the real line $\mathbb{R}$). For $H\in(0,1)$, a centered Gaussian process $B=(B_{t})_{t\in I}$ with covariance 
\[
\mathbb{E}\{ B_{s}B_{t} \} = \frac{1}{2}( |s|^{2H} + |t|^{2H} - |s-t|^{2H} )\ (s,t\in I)
\]
is called a (standard) \textit{fractional Brownian motion} (resp. \textit{two-sided fractional Brownian motion}) \textit{with Hurst parameter} $H\in(0,1)$. In the following, we always assume $H\in(0,1/2)$.
\end{definition}

As we explained in Section 1, we assume that the continuously observed data $(X_{t}^{\theta})_{t\in[0,T]}$ is available and is generated by the stochastic differential equation (\ref{eq1}). We impose the following assumptions on the coefficients in (\ref{eq1}).

\begin{assumption} \label{assumption1}
Let the parameter space $\Theta$ be an open subset of $\mathbb{R}^{m}$. We assume that the initial value $X_{0}$ is a deterministic constant $x_{0}\in\mathbb{R}$ and the diffusion coefficient $\sigma$ is not equal to zero.

Furthermore, we impose the following conditions on the drift coefficient $a\colon \mathbb{R}\times\Theta \to \mathbb{R}$.

\begin{enumerate}
\item[(A1)] All the partial derivatives of $a$ appearing below exist. 

\item[(A2)] For each $\theta\in\Theta$, there exists a positive constant $\alpha=\alpha(\theta)>0$ such that
\[
-\alpha^{-1} \leq \partial_{x}a(x,\theta) \leq -\alpha
\]
for all $x\in\mathbb{R}$.

\item[(A3)] For each $\theta\in\Theta$, there are a positive constant $C = C(\theta)>0$ and a nonnegative integer $p = p(\theta)$ such that 
\[
 |\partial_{\theta}a(x,\theta)|  \leq C(1+|x|^{p}). 
\]
and
\[
|\partial_{\theta}a(x,\theta) - \partial_{\theta}a(x,\theta^{\prime})| \leq C (1+ |x|^{p} ) | \theta-\theta^{\prime} |
\]
for all $x\in\mathbb{R}$ and $\theta^{\prime}\in \Theta$.

\item[(A4)] For each $\theta\in\Theta$, the function $\partial_{x}\partial_{\theta}a(x,\theta)$ is uniformly bounded in $x$.
\end{enumerate}
\end{assumption}

Without loss of generality, we can assume 
\begin{align}\label{ineq23}
\sup_{x\in\mathbb{R}}|(\partial_{x}a)(x,\theta)| + \sup_{x\in\mathbb{R}}|(\partial_{x}\partial_{\theta}a)(x,\theta)| \leq \alpha^{-1},
\end{align}
where $\alpha$ is from the condition (A2). In the following, we assume (\ref{ineq23}).

\begin{remark}
Since $x\mapsto a(x,\theta)$ is assumed to be differentiable, the condition (A2) is equivalent to the \textit{one-sided dissipative Lipshcitz condition} plus the uniform boundedness of $\partial_{x}a(\cdot,\theta)$. Recall that a function $f\colon\mathbb{R}^{n} \to \mathbb{R}^{n}$ satisfies the \textit{one-sided dissipative Lipshcitz condition} if there exists a constant $L>0$ such that
\[
\langle f(x)-f(y), x-y\rangle_{\mathbb{R}^{n}} \leq -L |x-y|^{2}
\]
for all $x$ and $y$ in $\mathbb{R}^{n}$. The one-sided dissipative Lipschitz condition is often imposed to ensure the ergodicity of the solution of the equation (\ref{eq1}), see \cite{cohen2011approximation, hairer2005ergodicity, garrido2009discretization} for example.
\end{remark}

\begin{notation}
Let $\mathbb{D}^{1,p}\ (p>0)$ denote the Sobolev space of random variables $X$ that are Malliavin differentiable and satisfy 
\[
\|X\|_{\mathbb{D}^{1,p}} := \mathbb{E}^{\ast}\{|X|^{p}\} + \mathbb{E}^{\ast}\{\|DX\|_{\mathcal{H}}^{p}\} < \infty.
\]
Here $\mathcal{H}$ is a real separable Hilbert space associated with a two-sided standard fBM. The precise definition of the space $\mathcal{H}$ is given in Section 3.
\end{notation}

We can ensure the existence and uniqueness of the solution for the equation (\ref{eq1}) under Assumption \ref{assumption1}. 

\begin{theorem} \label{theorem1}
Suppose that Assumption \ref{assumption1} is in force. Then there exists a (complete) probability space $(\Omega^{\ast},\mathcal{F}^{\ast},\mathbb{P}^{\ast})$ such that 
\begin{enumerate}
\item[(1)] there exists a two-sided standard fBM $B = (B_{t})_{t\in\mathbb{R}}$ on $(\Omega^{\ast},\mathcal{F}^{\ast},\mathbb{P}^{\ast})$, and
\item[(2)] the SDE (\ref{eq1}) has a unique pathwise solution $X^{x_{0},\theta} = (X^{x_{0},\theta}_{t})_{t\geq 0} $ for each $\theta\in\Theta$ and $x_{0}\in\mathbb{R}$ that is continuous and satisfies 
\begin{align}\label{ineq7}
 \sup_{t\geq0}\mathbb{E}^{\ast}\{ |X_{t}^{x_{0},\theta}|^{p} \} < \infty
\end{align}
for all $p>0$.
\end{enumerate}

Furthermore, for each $\theta\in\Theta$, there exists a unique stationary stochastic process $\bar{X}^{\theta} = (\bar{X}^{\theta}_{t})_{t\in\mathbb{R}}$ on $(\Omega^{\ast},\mathcal{F}^{\ast},\mathbb{P}^{\ast})$ with the following properties.
\begin{enumerate}
\item[(3)] The process $\bar{X}^{\theta}$ satisfies
\begin{align}\label{eq18}
\bar{X}^{\theta}_{t}(\omega) - \bar{X}^{\theta}_{s}(\omega) = \int_{s}^{t} a(\bar{X}^{\theta}_{r}(\omega),\theta)\, dr + \sigma( B_{t}(\omega) - B_{s}(\omega) )
\end{align}
for all $s<t$ and $\omega\in\Omega^{\ast}$. 

\item[(4)] For any $t\in\mathbb{R}$, the random variable $\bar{X}^{\theta}_{t}$ is in $\bigcap_{p>0}\mathbb{D}^{1,p}$, and its Malliavin derivative $D\bar{X}^{\theta}_{t}$ is given by
\begin{align}\label{ineq2}
(D_{\cdot}\bar{X}^{\theta}_{t})(\omega) = \sigma\exp\left(\int_{\cdot}^{t} \partial_{x}a(\bar{X}^{\theta}_{r}(\omega),\theta)\, dr \right)\mathbf{1}_{(-\infty,t]}(\cdot).
\end{align}
\end{enumerate}
\end{theorem}

The proof of Theorem \ref{theorem1} is given in Section 3.

\begin{notation}
Let $ C[0,T] $ be the space of $\mathbb{R}$-valued continuous functions on $[0,T]$ with the sup-norm $\| \cdot \|_{\infty}$ and $\pi = (\pi_{t})_{t\in[0,T]} $ be the canonical process, i.e., $ \pi_{t}(x) = x(t) $. The $\sigma$-field generated by $(\pi_{s})_{s\in[0,t]}$ is denoted by $\mathcal{B}_{t}$. Note that $ \mathcal{B}_{T} $ coincides with the Borel $\sigma$-field generated by the sup-norm. 
\end{notation}

Thanks to Theorem \ref{theorem1}, we can consider the probability distribution on $(C[0,T],\mathcal{B}(C[0,T]))$ induced by the solution of the equation (\ref{eq1}). Let us consider the family of statistical experiments $(\mathcal{E}_{T})_{T>0}$, where 
\[
\mathcal{E}_{T} = (C[0,T],\mathcal{B}_{T},(\mu_{\theta}^{T})_{\theta\in\Theta}).
\]
Here we consider the probability measures $(\mu_{\theta}^{T})_{\theta\in\Theta}$ on $(C[0,T],\mathcal{B}_{T})$ satisfying the following properties.
\begin{enumerate}
\item[(1)] The process $B^{\theta} = (B^{\theta}_{t})_{t\in[0,T]}$ defined by
\[
B_{t}^{\theta}(x) = \sigma^{-1} \left( \pi_{t}(x) - \pi_{0}(x) - \int_{0}^{t} a(\pi_{s}(x),\theta)\, ds \right)
\]
is a standard fBM on $[0,T]$ under the measure $\mu_{\theta}^{T}$.

\item[(2)]It holds that $\mu_{\theta}^{T}\{ \pi_{0} = x_{0} \} = 1$.
\end{enumerate}  

\begin{notation} Let $m$ be a positive integer. For a vector $u\in\mathbb{R}^{m}$, we denote the transpose of $u$ by $u^{\star}$. Let $\mu$ be in $\mathbb{R}^{m}$ and $\Sigma$ be a positive semi-definite $m\times m$ matrix. The $m$-dimensional normal distribution with mean $\mu$ and variance $\Sigma$ is denoted by $\mathcal{N}_{m}(\mu,\Sigma)$. 
\end{notation}

The aim of this paper is to prove \textit{local asymptotic normality} (LAN) of the probability measures $(\mu_{\theta}^{T})_{\theta\in\Theta}$. Let us briefly recall the definition of LAN. 

\begin{definition}
A family $( \mu_{\theta}^{T} )_{\theta\in\Theta}$ is called \textit{locally asymptotically normal} (\textit{LAN}) at a point $\theta \in \Theta$ if there exists some nondegenerate $m\times m$ matrix $\varphi_{T}(\theta)$ such that for any $u\in\mathbb{R}^{m}$ the likelihood ratio process
\[
Z_{\theta}^{T}(u) = \frac{d\mu^{T}_{\theta+\varphi_{T}(\theta)u}}{d\mu^{T}_{\theta}}
\]
can be represented as 
\[
Z_{\theta}^{T}(u) = \exp\left( u^{\star}\Delta_{\theta}^{T} - \frac{1}{2}u^{\star} I(\theta)u + r^{T}_{\theta,u} \right),
\]
where the matrix $I(\theta)$ is positive definite $m\times m$ matrix (the Fisher information matrix), the random variable $\Delta_{\theta}^{T} $ converges in distribution (with respect to $\mu_{\theta}^{T}$) to $\mathcal{N}_{m}(0,I(\theta))$ as $T\to\infty$, and $r^{T}_{\theta,u}$ satisfies 
\[
\lim_{T\to\infty}\mu_{\theta}^{T}\{ | r^{T}_{\theta,u} | > \epsilon \} = 0
\]
for any $\epsilon>0$. 
\end{definition}

Before stating our main theorem, we define a stochastic process which appears in the expression of the likelihood ratio process. 
\begin{definition}
Let $\beta_{t}(\theta)$ denote the process
\begin{align*}
\beta_{t}(\theta) &= d_{H}^{-1}\Gamma(1/2-H)^{-1} \sigma^{-1} t^{H-1/2} \int_{0}^{t} (t-s)^{-1/2-H}s^{1/2-H}a(\pi_{s},\theta)\, ds \\
&= \bar{d}_{H}^{-1}\sigma^{-1} t^{H-1/2} \int_{0}^{t} (t-s)^{-1/2-H}s^{1/2-H}a(\pi_{s},\theta)\, ds,
\end{align*}
where we set $d_{H} = \sqrt{\frac{2H\Gamma(3/2-H)\Gamma(H+1/2)}{\Gamma(2-2H)}}$ and $\bar{d}_{H}=\Gamma(1/2-H)d_{H}$. 
\end{definition}

\begin{remark}
\begin{asparaenum}
\item We have 
\[
\bar{d}_{H} = \sqrt{2HB(3/2-H,1/2-H)B(1/2+H,1/2-H)}.
\]
\item The process $\beta_{t}(\theta)(X^{\theta})$ coincides with the process $Q_{t}$ defined in Proposition 1 of \cite{tudor2007statistical} up to constant multiples.
\end{asparaenum}
\end{remark}

\begin{notation}
We denote the $n\times n$ identity matrix and the $p\times q$ zero matrix by $J_{n}$ and $\mathbf{0}_{p,q}$.
\end{notation}

Here is our main result.

\begin{theorem} \label{theorem2}
Suppose that Assumption \ref{assumption1} is in force. Furthermore, we assume $H\in(1/4,1/2)$ in (2).
\begin{enumerate}
\item[(1)] The family $(\mu_{\theta}^{T})_{\theta\in\Theta}$ is mutually absolutely continuous and its Radon-Nikodym derivative is given by 
\begin{align}\label{eq9}
\frac{d\mu_{\theta^{\prime}}^{T}}{d\mu_{\theta}^{T}} = \exp \left( \int_{0}^{T} (\beta_{t}(\theta^{\prime}) - \beta_{t}(\theta))\,dW^{\theta}_{t} - \frac{1}{2} \int_{0}^{T} ( \beta_{t}(\theta^{\prime}) - \beta_{t}(\theta) )^{2}\,dt \right)
\end{align}
for all $\theta, \theta^{\prime} \in \Theta$. Here the process $W^{\theta}$ is a $(\mathcal{B}_{t})$-Brownian motion under the measure $\mu_{\theta}^{T}$ (for the precise definition of $W^{\theta}$, see (\ref{eq26}) below).

\item[(2)] Assume that 
\begin{itemize}
\item $\mathbb{E}^{\ast}\{ \partial_{\theta_{i}} a(\bar{X}^{\theta}_{0},\theta) \} = 0 $ for $i=1,\ldots,m_{0}(\theta)$, and
\item $ \mathbb{E}^{\ast}\{ \partial_{\theta_{i}} a(\bar{X}^{\theta}_{0},\theta) \} \neq 0 $ for $i=m_{0}(\theta)+1,\ldots,m$. 
\end{itemize}

Then the family $(\mu_{\theta}^{T})_{\theta\in\Theta}$ is LAN at a point $ \theta\in\Theta $, with a normalizing matrix 
\[
\varphi_{T}(\theta) = 
\left(\begin{array}{cc}
T^{-1/2}J_{m_{0}(\theta)} & \mathbf{0}_{m_{0}(\theta),m-m_{0}(\theta)} \\
\mathbf{0}_{m-m_{0}(\theta),m_{0}(\theta)} & T^{-(1-H)} J_{m-m_{0}(\theta)}
\end{array} \right)
\]
and the Fisher information matrix $I(\theta)=(I_{i,j}(\theta))_{i,j=1\ldots,m}$, where
\begin{align*}
&I_{i,j}(\theta) \\
&= \sigma^{-2}\bar{d}_{H}^{-2}\int_{0}^{\infty}dr\int_{0}^{\infty}du\ r^{-H-1/2}u^{-H-1/2} \mathbb{E}^{\ast}\{ \partial_{\theta_{i}}a(\bar{X}^{\theta}_{r},\theta)\partial_{\theta_{j}}a(\bar{X}^{\theta}_{u},\theta) \}
\end{align*}
for $i,j=1,\ldots,m_{0}(\theta)$, 
\begin{align*}
I_{i,j}(\theta) = \sigma^{-2}d_{H}^{\prime}\mathbb{E}^{\ast}\{\partial_{\theta_{i}}a(\bar{X}^{\theta}_{0},\theta)\}\mathbb{E}^{\ast}\{\partial_{\theta_{j}}a(\bar{X}^{\theta}_{0},\theta)\} 
\end{align*}
for $i,j=m_{0}(\theta)+1,\ldots,m$, and 
\[
I_{i,j}(\theta) = 0
\]
else.
Here the constant $d_{H}^{\prime}$ is 
\[
d_{H}^{\prime} = \bar{d}_{H}^{-2}B(1/2-H, 3/2-H)^{2}(2-2H)^{-1}.
\]
\end{enumerate}
\end{theorem}

\begin{remark}
Theorem \ref{theorem2} also states that the double integral
\[
\int_{0}^{\infty}dr\int_{0}^{\infty}du\ r^{-H-1/2}u^{-H-1/2} \mathbb{E}^{\ast}\{ \partial_{\theta_{i}}a(\bar{X}^{\theta}_{r},\theta)\partial_{\theta_{j}}a(\bar{X}^{\theta}_{u},\theta) \}
\]
is well-defined. More precisely, we will prove that
\[
\int_{0}^{\infty}dr\int_{0}^{\infty}du\ r^{-H-1/2}u^{-H-1/2} |\mathbb{E}^{\ast}\{ \partial_{\theta_{i}}a(\bar{X}^{\theta}_{r},\theta)\partial_{\theta_{j}}a(\bar{X}^{\theta}_{u},\theta) \}| < \infty
\]
holds.
\end{remark}

The proof of the first part of Theorem \ref{theorem2} is given in Section 4, and the second part in Section 5.

\begin{remark}\label{remark2}
Let us consider the case where the parameter space is one-dimensional, the diffusion coefficient $\sigma$ equals to $1$, and the drift coefficient $a(\theta,x)$ is of the form $\theta b(x)$ for some function $b$. In this case, we can explicitly calculate the MLE $\hat{\theta}_{T}$ for the true parameter $\theta$ by the formula (\ref{eq9}). An explicit calculation yields
\[
\varphi_{T}(\theta)^{-1} (\hat{\theta}_{T}-\theta) = \frac{ \varphi_{T}(\theta)\int_{0}^{T} (\partial_{\theta}\beta)_{t}(\theta)\, dW_{t}^{\theta}  }{ \varphi_{T}(\theta)^{2} \int_{0}^{T} (\partial_{\theta}\beta)_{t}(\theta)^{2}\, dt }.
\]
As a consequence of (the proof of) Theorem \ref{theorem2}, we have
\begin{align} \label{eq45}
\varphi_{T}(\theta)^{-1} (\hat{\theta}_{T}-\theta) \to^{d} \mathcal{N}_{1}(0,I(\theta)^{-1})
\end{align}
as $T\to\infty$, where $\varphi_{T}(\theta)=T^{-1/2}$ and
\[
I(\theta) = \sigma^{-2} \bar{d}_{H}^{-2}\int_{0}^{\infty}dr\int_{0}^{\infty}du\ r^{-H-1/2}u^{-H-1/2} \mathbb{E}^{\ast}\{ b(\bar{X}^{\theta}_{r})b(\bar{X}^{\theta}_{u})\}
\]
in the case of $ \mathbb{E}^{\ast}\{ b(\bar{X}^{\theta}_{0}) \} = 0 $, and $\varphi_{T}(\theta)=T^{H-1}$ and
\[
I(\theta) = \sigma^{-2}d_{H}^{\prime}\mathbb{E}^{\ast}\{b(\bar{X}^{\theta}_{0})\}^{2} 
\]
in the case of $ \mathbb{E}^{\ast}\{ b(\bar{X}^{\theta}_{0}) \} \neq 0 $. In particular, the MLE defined in \cite{tudor2007statistical} is asymptotically normal when $H\in(1/4,1/2)$.

Using this example, we try to explain the idea of the proof of the local asymptotic normality in the Section 2.1.

\end{remark}

In the case of fractional Ornstein-Uhlenbeck process (i.e., the case where $a(x,\theta)=-\theta x$ for $\theta\in(0,\infty)$), we can calculate the Fisher information explicitly.

\begin{proposition}\label{proposition4}
Suppose that $a(x,\theta)=-\theta x$ for $\theta\in(0,\infty)$. Then Assumption \ref{assumption1} is fulfilled and hence Theorem \ref{theorem2} holds with $\varphi_{T}(\theta) = T^{-1/2}$ and 
\[
I(\theta) = \frac{1}{2\theta}.
\]
\end{proposition}

The proof of Proposition \ref{proposition4} is given in Section 2.2.

\begin{remark}
By Remark \ref{remark2} and Proposition \ref{proposition4}, we see that the asymptotic variance of the MLE is $2\theta$. This is consistent with the asymptotic behavior of the MLE observed in Proposition 3 of \cite{tudor2007statistical}.
\end{remark}

\begin{notation}
 Let $a$ and $b$ be real numbers. We denote $a\lesssim b$ if there is a positive constant $C$ which depends only on $H$, $\sigma$, $\theta$ and $\alpha$ (see Assumption \ref{assumption1}) such that
\begin{align}\label{ineq24}
a \leq C b
\end{align}
holds. If $a$ and $b$ are random variables, then we denote $a\lesssim b$ when the inequality (\ref{ineq24}) holds in a.s. sense.
\end{notation}

\subsection{The idea for the proof of the LAN property} \label{section2.1}

Using the example in Remark \ref{remark2}, we try to explain the idea of the proof of the local asymptotic normality. 

To prove (\ref{eq45}), it suffices to show the limit 
\begin{align} \label{limit7}
 \varphi_{T}(\theta)^{2}\int_{0}^{T}(\partial_{\theta}\beta)_{t}(\theta)(X^{x_{0},\theta})^{2}\,dt \to I(\theta)
\end{align}
in probability as $T\to\infty$ by the martingale central limit theorem. \textit{We remark that showing this kind of limit is the essential difficulty when we prove the local asymptotic normality.} Indeed, as we shall see later, the likelihood ratio field $Z_{\theta}^{T}(u) =  (d\mu^{T}_{\theta+\varphi_{T}(\theta)u} / d\mu^{T}_{\theta}) $ can be decomposed as 
\begin{align*}
\log Z_{\theta}^{T}(u) &= u^{\star}\varphi_{T}(\theta)\int_{0}^{T}\partial_{\theta}\beta_{t}(\theta)\ dW_{t}^{\theta} \\
&\quad - \frac{1}{2} u^{\star} \varphi_{T}(\theta) \int_{0}^{T} \partial_{\theta}\beta_{t}(\theta)( \partial_{\theta}\beta_{t}(\theta) )^{\star}\,dt\ \varphi_{T}(\theta) u + o_{p}(1).
\end{align*}

As in Theorem \ref{theorem1}, there exists the stationary solution $\bar{X}^{\theta}$ of the equation (\ref{eq1}). Furthermore, we can prove $|X^{x_{0},\theta}_{t}-\bar{X}^{\theta}_{t}|$ converges exponentially to zero as $t\to\infty$ under Assumption \ref{assumption1}. As a consequence, the quantities $\varphi_{T}(\theta)^{2}\int_{0}^{T}(\partial_{\theta}\beta)_{t}(\theta)(X^{x_{0},\theta})^{2}\,dt$ and  $\varphi_{T}(\theta)^{2}\int_{0}^{T}(\partial_{\theta}\beta)_{t}(\theta)(\bar{X}^{\theta})^{2}\,dt$ are \textit{asymptotically equivalent} under $\mathbb{P}^{\ast}$. In particular, we reduce (\ref{limit7}) to the limit
\begin{align} \label{limit8}
I^{T}(\theta) := \varphi_{T}(\theta)^{2}\int_{0}^{T}(\partial_{\theta}\beta)_{t}(\theta)(\bar{X}^{\theta})^{2}\,dt \to I(\theta)
\end{align}
in probability as $T\to\infty$.

First let us consider the case of $ \mathbb{E}^{\ast}\{ b(\bar{X}^{\theta}_{0})\} = 0 $. In this case, $\partial_{\theta}\beta_{t}(\theta)(\bar{X}^{\theta})$ can be decomposed as follows:
\begin{align*}
\partial_{\theta}\beta_{t}(\theta)(\bar{X}^{\theta}) &= \sigma^{-1}\bar{d}_{H}^{-1} t^{H-1/2}\int_{0}^{t} (t-r)^{-H-1/2}r^{1/2-H} b(\bar{X}_{r}^{\theta})\, dr \\
\nonumber &= \sigma^{-1}\bar{d}_{H}^{-1} \int_{0}^{t} r^{-1/2-H} b(\bar{X}_{t-r}^{\theta}) \, dr \\
\nonumber &\quad +\sigma^{-1}\bar{d}_{H}^{-1} \int_{0}^{t} r^{-1/2-H} \left\{ \left( 1- \frac{r}{t} \right)^{1/2-H} -1 \right\} b(\bar{X}_{t-r}^{\theta}) \, dr\\
&=: \gamma_{t}^{2}(\theta) + \gamma_{t}^{3}(\theta)
\end{align*}
(the term $\gamma_{t}^{1}(\theta)$ appears when $ \mathbb{E}^{\ast}\{ b(\bar{X}^{\theta}_{0})\} \neq 0 $). Hence 
\begin{align*}
I^{T}(\theta) &= \sum_{p,q=2,3} T^{-1}\int_{0}^{T}\gamma_{t}^{p}(\theta)\gamma_{t}^{q}(\theta)\,dt \\
&=: \sum_{p,q=2,3} J_{p,q}^{T}(\theta).
\end{align*}
holds (note that $\varphi_{T}(\theta) = T^{-1/2}$ when $ \mathbb{E}^{\ast}\{ b(\bar{X}^{\theta}_{0})\} = 0 $). We expect that the term $J_{2,2}^{T}(\theta) =T^{-1}\int_{0}^{T}\gamma_{t}^{2}(\theta)\gamma_{t}^{2}(\theta)\,dt$ only affects the limit and the other terms are negligible. The problem is that we can not apply the ordinary ergodic theorem to $J_{p,q}^{T}(\theta)$. Instead of using the ergodic theorem, we calculate the limit of $J_{p,q}^{T}(\theta)$ in two steps: first we calculate the limit of the expectation 
\[
\lim_{T\to\infty}\mathbb{E}^{\ast}\{ J_{p,q}^{T}(\theta)\}
\]
and second show that 
\[
\lim_{T\to\infty}\mathbb{E}^{\ast}\{ | J_{p,q}^{T}(\theta) - \mathbb{E}^{\ast}\{ J_{p,q}^{T}(\theta)\} |^{2} \} =0.
\]

Let us focus on the limit $\lim_{T\to\infty}\mathbb{E}^{\ast}\{ J_{2,2}^{T}(\theta)\}$. The expectation $\mathbb{E}^{\ast}\{ J_{2,2}^{T}(\theta)\}$ can be represented as
\begin{align*}
&\mathbb{E}^{\ast}\{ J_{2,2}^{T}(\theta)\} \\
&= \sigma^{-2}\bar{d}_{H}^{-2}T^{-1} \int_{0}^{T}dt\int_{0}^{t}dr\int_{0}^{t}du\, r^{-H-1/2}u^{-H-1/2}\mathbb{E}^{\ast}\{ b(\bar{X}^{\theta}_{|r-u|})b(\bar{X}^{\theta}_{0}) \} \\
&= \sigma^{-2}\bar{d}_{H}^{-2} \int_{0}^{1}dt\int_{0}^{tT}dr\int_{0}^{tT}du\, r^{-H-1/2}u^{-H-1/2}\mathbb{E}^{\ast}\{ b(\bar{X}^{\theta}_{|r-u|})b(\bar{X}^{\theta}_{0}) \}
\end{align*}
(here we used the stationarity of $\bar{X}^{\theta}$). Let us set 
\[
c(t)=\mathrm{Cov}\{b(\bar{X}^{\theta}_{t}),b(\bar{X}^{\theta}_{0})\}.
\]
It is shown in \cite{cheridito2003fractional} that $c(t) = O(t^{2H-2})$ as $t\to\infty$ if $b(x)=-x$. It is plausible to expect that $c(t)$ decays fast as $t\to\infty$ even in the nonlinear case. In fact, we can prove the estimate
\begin{align}\label{ineq25}
|c(t)| \lesssim t^{H-3/2}
\end{align}
for $t\geq1$ (Lemma \ref{corollary1}). In contrast to the fractional Ornstein-Uhlenbeck case, where the stationary solution has the explicit expression, we can not calculate $c(t)$ directly. \textit{However, the Malliavin derivative of the stationary solution has the explicit expression (\ref{ineq2}).} To obtain the inequality (\ref{ineq25}), we use an integration by parts formula from Malliavin calculus:
\[
c(t) = \mathbb{E}^{\ast} \{ \langle Db(\bar{X}^{\theta}_{t}) , -DL^{-1}b(\bar{X}^{\theta}_{0}) \rangle_{\mathcal{H}} \} = \mathbb{E}^{\ast} \{  \dot{b}(\bar{X}^{\theta}_{t}) \langle D\bar{X}^{\theta}_{t} , -DL^{-1}b(\bar{X}^{\theta}_{0}) \rangle_{\mathcal{H}} \}.
\]
Therefore it is necessary to investigate Malliavin calculus for the stationary solution $\bar{X}^{\theta}$, and this is done in Section 3.

Once we obtain the inequality (\ref{ineq25}), it is easy to show that 
\[
\int_{0}^{\infty}dr\int_{0}^{\infty}du\,r^{-H-1/2}u^{-H-1/2}|c(|r-u|)| < \infty
\]
and the dominated convergence theorem yields $\lim_{T\to\infty}\mathbb{E}^{\ast}\{J^{T}_{2,2}(\theta)\} = I(\theta)$. The dominated converegence theorem also shows $\lim_{T\to\infty}\mathbb{E}^{\ast}\{J^{T}_{3,3}(\theta)\}=0$. In particular, we have $\lim_{T\to\infty}\mathbb{E}^{\ast}\{| I^{T}(\theta) - J_{2,2}^{T}(\theta) |\}=0$ by H\"{o}lder's inequality.

The remaining problem is to show the limit
\begin{align}\label{limit9}
\lim_{T\to\infty}\mathbb{E}^{\ast}\{|J_{2,2}^{T}(\theta) - \mathbb{E}^{\ast}\{J_{2,2}^{T}(\theta)\}|^{2}\} = 0.
\end{align}
We again use Malliavin calculus to show the limit (\ref{limit9}): we have
\begin{align}
\mathbb{E}^{\ast}\{|J_{2,2}^{T}(\theta) - \mathbb{E}^{\ast}\{J_{2,2}^{T}(\theta)\}|^{2}\} \leq \mathbb{E}^{\ast}\{\|DJ_{2,2}^{T}(\theta)\|^{2}_{\mathcal{H}}\}
\end{align}
by the Poincar\'{e} inequality. We can prove $\lim_{T\to\infty}\mathbb{E}^{\ast}\{\|DJ_{2,2}^{T}(\theta)\|^{2}_{\mathcal{H}}\}=0$ if $H\in(1/4,1/2)$ (Lemma \ref{lemma10}). \textit{We remark that the assumption $H>1/4$ is used only here.} After all, we obtain $\lim_{T\to\infty}\mathbb{E}^{\ast}\{ | I^{T}(\theta) - I(\theta) | \} = 0$ when $H\in(1/4,1/2)$. Hence the limit (\ref{limit8}) holds when $\mathbb{E}^{\ast}\{b(\bar{X}^{\theta}_{0})\}=0$.

Next we consider the case where $ \mathbb{E}^{\ast}\{ b(\bar{X}^{\theta}_{0})\} \neq 0 $. In this case, $\partial_{\theta}\beta_{t}(\theta)(\bar{X}^{\theta})$ can be decomposed as:
\begin{align*}
&\partial_{\theta}\beta_{t}(\theta)(\bar{X}^{\theta}) \\
&= \sigma^{-1}\bar{d}_{H}^{-1}t^{H-1/2}\int_{0}^{t} (t-r)^{-H-1/2}r^{1/2-H} b(\bar{X}_{r}^{\theta})\, dr \\
&= \sigma^{-1}\bar{d}_{H}^{-1}B(3/2-H,1/2-H)\mathbb{E}^{\ast}\{b(\bar{X}^{\theta}_{0})\}t^{1/2-H}\\
&\quad +\sigma^{-1}\bar{d}_{H}^{-1} \int_{0}^{t} r^{-1/2-H} (b(\bar{X}_{t-r}^{\theta})-\mathbb{E}^{\ast}\{b(\bar{X}^{\theta}_{0})\}) \, dr \\
&\quad +\sigma^{-1}\bar{d}_{H}^{-1} \int_{0}^{t} r^{-1/2-H} \left\{ \left( 1- \frac{r}{t} \right)^{1/2-H} -1 \right\} (b(\bar{X}_{t-r}^{\theta}) -\mathbb{E}^{\ast}\{b(\bar{X}^{\theta}_{0})\})\, dr\\
&=: \gamma_{t}^{1}(\theta) + \gamma_{t}^{2}(\theta) + \gamma_{t}^{3}(\theta).
\end{align*}
Hence 
\begin{align*}
I^{T}(\theta) &= \sum_{p,q=1,2,3} T^{2H-2}\int_{0}^{T}\gamma_{t}^{p}(\theta)\gamma_{t}^{q}(\theta)\,dt \\
&=: \sum_{p,q=1,2,3} J_{p,q}^{T}(\theta)
\end{align*}
(recall that $\varphi_{T}(\theta)=T^{H-1}$ when $\mathbb{E}^{\ast}\{b(\bar{X}^{\theta}_{0})\}\neq0$). A straightforward calculation yields $J_{1,1}^{T}(\theta)=I(\theta)$. As we saw in the previous case, we have
\[
\mathbb{E}^{\ast}\int_{0}^{T}\gamma_{t}^{2}(\theta)^{2}\,dt = \mathbb{E}^{\ast}\int_{0}^{T}\gamma_{t}^{3}(\theta)^{2}\,dt = o(T^{2H-2}).
\]
H\"{o}lder's inequality yields $\lim_{T\to\infty}\mathbb{E}^{\ast}\{|I^{T}(\theta)-I(\theta)|\} = 0$. Hence the limit (\ref{limit8}) also holds when $\mathbb{E}^{\ast}\{b(\bar{X}^{\theta}_{0})\}\neq0$.

\subsection{Proof of Proposition \ref{proposition4}}

Note that the stationary solution is given by
\[
\bar{X}^{\theta}_{t} = \sigma\int_{-\infty}^{t} e^{-\theta(t-u)}\,dB_{u}
\]
and hence
\[
\mathbb{E}^{\ast}\{\partial_{\theta}a(\bar{X}^{\theta}_{0},\theta)\} = -\mathbb{E}^{\ast}\{ \bar{X}^{\theta}_{0} \} = 0
\]
holds (see \cite{cheridito2003fractional}). In particular, $\varphi_{T}(\theta) = T^{-1/2} $.
It is also known that 
\begin{align*}
\mathbb{E}^{\ast}\{\bar{X}^{\theta}_{r}\bar{X}^{\theta}_{u}\} &= \sigma^{2}\frac{\Gamma(2H+1)\sin(\pi H)}{2\pi}\int_{-\infty}^{\infty}dx \, e^{i(r-u)x}\frac{|x|^{1-2H}}{\theta^{2}+x^{2}} 
\end{align*}
holds (see \cite{cheridito2003fractional}). We set 
\[
\mathtt{I} = \int_{0}^{\infty}dr\int_{0}^{\infty}du\,r^{-H-1/2}u^{-H-1/2} \int_{-\infty}^{\infty}dx \, e^{i(r-u)x}\frac{|x|^{1-2H}}{\theta^{2}+x^{2}}
\]
and
\[
\mathtt{I}_{T} = \int_{0}^{T}dr\int_{0}^{T}du\,r^{-H-1/2}u^{-H-1/2} \int_{-\infty}^{\infty}dx \, e^{i(r-u)x}\frac{|x|^{1-2H}}{\theta^{2}+x^{2}}.
\]
First we calculate the double integral $\mathtt{I}$. Let us consider the following decomposition:
\begin{align*}
\mathtt{I}_{T} &= 2 \int_{0}^{\infty}dx \, \frac{|x|^{1-2H}}{\theta^{2}+x^{2}} \int_{0}^{T}dr \int_{0}^{T}du\, r^{-H-1/2}u^{-H-1/2}\cos((r-u)x) \\
&= 4 \int_{0}^{\infty}dx \, \frac{|x|^{1-2H}}{\theta^{2}+x^{2}} \int_{0}^{T}du\,\cos(ux)  \int_{u}^{T}dr\, r^{-H-1/2}(r-u)^{-H-1/2} \\
&= 4 \int_{0}^{\infty}dx \, \frac{|x|^{1-2H}}{\theta^{2}+x^{2}} \int_{0}^{T}du\,\cos(ux)  \int_{u}^{\infty}dr\, r^{-H-1/2}(r-u)^{-H-1/2} \\
&\quad -4 \int_{0}^{\infty}dx \, \frac{|x|^{1-2H}}{\theta^{2}+x^{2}} \int_{0}^{T}du\,\cos(ux)  \int_{T}^{\infty}dr\, r^{-H-1/2}(r-u)^{-H-1/2} \\ 
&= 4\int_{1}^{\infty}dr\,r^{-H-1/2}(r-1)^{-H-1/2}\int_{0}^{\infty}dx\,\frac{x^{1-2H}}{\theta^{2}+x^{2}}\int_{0}^{\infty}du\,u^{-2H}\cos(ux) \\
&\quad -4 \int_{1}^{\infty}dr\,r^{-H-1/2}(r-1)^{-H-1/2}\int_{0}^{\infty}dx\,\frac{x^{1-2H}}{\theta^{2}+x^{2}}\int_{T}^{\infty}du\,u^{-2H}\cos(ux) \\
&\quad -4 \int_{0}^{\infty}dx \, \frac{|x|^{1-2H}}{\theta^{2}+x^{2}} \int_{0}^{T}du\,\cos(ux)  \int_{T}^{\infty}dr\, r^{-H-1/2}(r-u)^{-H-1/2} \\
&=: \mathtt{I}^{\prime} - \mathtt{I}^{1}_{T} - \mathtt{I}^{2}_{T}.
\end{align*}

\begin{lemma}\label{lemma16}
We have $\lim_{T\to\infty}|\mathtt{I}_{T}^{1}| = \lim_{T\to\infty} |\mathtt{I}_{T}^{2}| =0 $.
\end{lemma}
\begin{proof}
The integration by parts formula yields the following decompositions:
\begin{align*}
&\mathtt{I}^{1}_{T} \\
&= 4\int_{1}^{\infty}dr\,r^{-H-1/2}(r-1)^{-H-1/2}\int_{0}^{\infty}dx\,\frac{x^{1-2H}}{\theta^{2}+x^{2}}\int_{T}^{\infty}du\,u^{-2H}\cos(ux)\\
&= -4T^{-2H} \int_{1}^{\infty}dr\,r^{-H-1/2}(r-1)^{-H-1/2}\int_{0}^{\infty}dx\,\frac{x^{-2H}}{\theta^{2}+x^{2}}\sin(Tx) \\
&\quad +8H\int_{1}^{\infty}dr\,r^{-H-1/2}(r-1)^{-H-1/2}\int_{0}^{\infty}dx\,\frac{x^{-2H}}{\theta^{2}+x^{2}}\int_{T}^{\infty}du\,u^{-2H-1}\sin(ux) \\
&= -\mathtt{I}^{1,1}_{T} + \mathtt{I}^{1,2}_{T}
\end{align*}
and 
\begin{align*}
\mathtt{I}^{2}_{T} &= 4\int_{0}^{\infty}dx\,\frac{x^{1-2H}}{\theta^{2}+x^{2}}\int_{T}^{\infty}dr\,r^{-H-1/2}\int_{0}^{T}du\,\cos(ux)(r-u)^{-H-1/2} \\
&= 4\int_{0}^{\infty}dx\,\frac{x^{-2H}}{\theta^{2}+x^{2}}\int_{T}^{\infty}dr\,r^{-H-1/2}(r-T)^{-H-1/2}\sin(Tx) \\
&\quad - 4(H+1/2)\int_{0}^{\infty}dx\,\frac{x^{-2H}}{\theta^{2}+x^{2}}\int_{T}^{\infty}dr\,r^{-H-1/2}\\ 
&\quad \times\int_{0}^{T}du\,(r-u)^{-H-3/2}\sin(ux) \\
&=: \mathtt{I}^{2,1}_{T} - \mathtt{I}^{2,2}_{T}.
\end{align*}
Using these representations, we can easily verify that 
\[
|\mathtt{I}_{T}^{i,j}| = O(T^{-2H})
\]
as $T\to\infty$ for $i,j=1,2$.
\end{proof}

By Lemma \ref{lemma16} above, we have $\lim_{T\to\infty}|\mathtt{I}_{T}^{1}| = \lim_{T\to\infty} |\mathtt{I}_{T}^{2}| =0 $ and hence $ \mathtt{I} = \lim_{T\to\infty}\mathtt{I}_{T} = \mathtt{I}^{\prime}$.
The integrals in $\mathtt{I}^{\prime}$ can be calculated explicitly:
\[
\int_{1}^{\infty}dr\,r^{-H-1/2}(r-1)^{-H-1/2} = B(1/2-H,2H)
\]
(see p.441 of \cite{zwillinger2012crc}), 
\[
\int_{0}^{\infty}du\,u^{-2H}\cos(ux) = x^{2H-1} \frac{\pi}{2\Gamma(2H)\cos(\pi H)}
\]
(see (613) in p.332 of \cite{zwillinger2012crc}) and
\[
\int_{0}^{\infty}dx\,\frac{1}{\theta^{2}+x^{2}} = \frac{\pi}{2\theta}
\]
(see (597) in p.330 of \cite{zwillinger2012crc}).
Therefore we have
\begin{align*}
\mathtt{I}^{\prime} &= 4 B(1/2-H,2H)\times\frac{\pi}{2\theta}\times\frac{\pi}{2\Gamma(2H)\cos(\pi H)}.
\end{align*}

Now let us turn to calculate the Fisher information $I(\theta)$:
\begin{align*}
I(\theta) &= \sigma^{-2} \bar{d}_{H}^{-2} \times \sigma^{2}\frac{\Gamma(2H+1)\sin(\pi H)}{2\pi} \times \mathtt{I}\\
&= \frac{1}{2\theta} \times \frac{B(1/2-H,2H)\sin(\pi H)\pi}{B(3/2-H,1/2-H)B(1/2+H,1/2-H)\sin(\pi(H+1/2))}.
\end{align*}
Euler's reflection formula gives
\[
I(\theta) = \frac{1}{2\theta} \times \frac{B(1/2-H,2H)\pi}{B(3/2-H,1/2-H)B(1/2+H,1/2-H)B(H,1-H)}.
\]
By applying the duplication formula after rewriting $I(\theta)$ using the Gamma function, we get
\begin{align*}
I(\theta) &= \frac{1}{2\theta} \times \frac{\Gamma(2H)\Gamma(2-2H)}{\Gamma(H)\Gamma(1-H)}\times\frac{\pi}{\Gamma(1/2+H)\Gamma(3/2-H)}\\
&= \frac{1}{2\theta} \times \frac{\Gamma(H+1/2)\Gamma(3/2-H)}{(\sqrt{\pi}2^{1-2H})(\sqrt{\pi}2^{1-(2-2H)})} \times \frac{\pi}{\Gamma(1/2+H)\Gamma(3/2-H)} \\
&= \frac{1}{2\theta}.
\end{align*}
This completes the proof.

\section{On the stationary solution of the equation (\ref{eq1})}

In this section, we investigate some properties of the stationary solution of the equtaion (\ref{eq1}). In particular, we provide the proof of Theorem \ref{theorem1}. First we specify the probability space $(\Omega^{\ast},\mathcal{F}^{\ast},\mathbb{P}^{\ast}) $ in Theorem \ref{theorem1}. 

Let $ \Omega^{\ast} = C_{0}(\mathbb{R})$ be the set of continuous function $\omega$ with $\omega(0)=0$. We consider the topology of compact convergence and the corresponding Borel $\sigma$-algebra on $ \Omega^{\ast} $. We denote this Borel $\sigma$-algebra as $ \mathcal{F}^{\ast}_{0} $. Then there exists a probability measure $\mathbb{P}^{\ast}_{0}$ on $(\Omega^{\ast}, \mathcal{F}^{\ast}_{0})$ such that the canonical process $\pi=(\pi_{t})_{t\in\mathbb{R}}$ is a two-sided fBM under $\mathbb{P}^{\ast}_{0}$. We define a $\sigma$-algebra $\mathcal{F}^{\ast}$ as the completion of $\mathcal{F}^{\ast}_{0}$ with respect to $\mathbb{P}^{\ast}_{0}$. The probability measure $\mathbb{P}^{\ast}_{0}$ can be naturally extended to the probability measure on $(\Omega^{\ast},\mathcal{F}^{\ast})$. This extension is denoted by $\mathbb{P}^{\ast}$. 

It is known that there exists a set $\Omega^{\ast}_{0}\in\mathcal{F}^{\ast}_{0}$ such that $\mathbb{P}^{\ast}_{0}\{\Omega_{0}^{\ast}\}=1$ and for each $\omega\in\Omega^{\ast}$
\[
|\pi_{t}(\omega)| \leq K(\omega)(1+|t|^{2})
\]
holds for all $t\in\mathbb{R}$, where $K(\omega)>0$ is a random constant. For the proof of this fact, see Lemma 3.3 of \cite{gess2011random}. We define $B_{t}\colon\Omega^{\ast}\to\mathbb{R}$ by $ B_{t}(\omega) = \pi_{t}(\omega)\mathbf{1}_{\Omega^{\ast}_{0}}(\omega) $ for each $t\in\mathbb{R}$. We set $ B=(B_{t})_{t\in\mathbb{R}} $.

\begin{remark}
Note that the process $B$ is also a two-sided fBM under $\mathbb{P}_{0}^{\ast}$ and $\mathbb{P}^{\ast}$. As is done in \cite{garrido2009discretization}, we would rather regard $B$ as the driving fBM than the canonical process $\pi$. 
\end{remark}

We start with showing that the equation (\ref{eq1}) has a unique continuous solution for a given initial condition $X_{0}$. 
The next proposition gives the proof of Theorem \ref{theorem1}(2).

\begin{proposition}
Suppose that Assumption \ref{assumption1} holds. Let $s$ be a real number and $\xi$ be a random variable on $(\Omega^{\ast},\mathcal{F}^{\ast})$. Then the equation
\[
X_{t} = \xi + \int_{s}^{t} a(X_{r},\theta)\,dr + \sigma (B_{t}-B_{s}),\ t\in[s,\infty)
\]
has a unique continuous solution $X^{\xi,\theta,s}=(X^{\xi,\theta,s}_{t})_{t\in[s,\infty)}$ for each $\omega\in\Omega^{\ast}$. Furthermore, if $\xi$ is a constant, then $X^{\xi,\theta,s}$ satisfies 
\begin{align}\label{ineq20}
\sup_{t\in[s,\infty)} \mathbb{E}^{\ast}\{ |X^{\xi,\theta,s}_{t}|^{p} \} < \infty.
\end{align}
\end{proposition}

\begin{proof}




The existence and uniqueness of the solution is due to a standard Picard iteration argument. Therefore we omit the proof.




The inequality (\ref{ineq20}) can be proved by the same argument as in the proof of Proposition 2.2 of \cite{neuenkirch2014least}.
\end{proof}

In order to investigate the properties of the stationary solution of the equtaion (\ref{eq1}), we use the theory of random dynamical systems and random attractors. In the sequel, we follow the terminologies of \cite{garrido2009discretization} for random dynamical systems and random attractors.

\begin{definition} \label{definition1}
Let $(\Omega,\mathcal{F},\mathbb{P})$ be a probability space.
\begin{enumerate}
\item Suppose that a family of transformations $\{ \vartheta_{t}\colon\Omega\to\Omega;t\in\mathbb{R} \}$ satisfies that
\begin{itemize} 
\item $(t,\omega)\mapsto\vartheta_{t}\omega$ is $(\mathcal{B}(\mathbb{R})\otimes\mathcal{F};\mathcal{F})$-measurable, 
\item $\vartheta_{0}\omega = \omega$ for all $\omega\in\Omega$,
\item $\vartheta_{t}\circ\vartheta_{s} = \vartheta_{t+s}$ for all $s,t\in\mathbb{R}$ and
\item $\mathbb{P}^{\vartheta_{t}} = \mathbb{P}$ for all $t\in\mathbb{R}$.
\end{itemize}
Then the quadruple $\vartheta=(\Omega,\mathcal{F},\mathbb{P},\{ \vartheta_{t};t\in\mathbb{R} \})$ is called a \textit{(continuous) metric dynamical system (MDS)}. 

\item A map $\phi\colon\mathbb{R}_{\geq0}\times\Omega\times\mathbb{R}^{d} \to \mathbb{R}^{d}$ is called a \textit{cocycle mapping} if
\begin{itemize} 
\item $\phi$ is $(\mathcal{B}(\mathbb{R}_{\geq0})\otimes\mathcal{F}\otimes\mathcal{B}(\mathbb{R}^{d});\mathcal{B}(\mathbb{R}^{d}))$-measurable,
\item $\phi(0,\omega,x)=x$, for all $\omega\in\Omega$ and $x\in\mathbb{R}^{d}$, and 
\item $\phi(t+s,\omega,x) = \phi(s,\vartheta_{t}\omega,\phi(t,\omega,x))$ for all $t,s\in\mathbb{R}_{\geq 0}$, $x\in\mathbb{R}^{d}$ and $\omega \in \Omega$. 
\end{itemize}

\item The pair $(\vartheta,\phi)$ of a (continous) MDS and a cocycle mapping is called a \textit{(continuous) random dynamical system (RDS)}.

\item A \textit{universe} $ \mathcal{D} $ is a collection of nonempty random sets $ (D(\omega))_{\omega\in\Omega} $ of $\mathbb{R}^{d}$ which is closed with respect to set inclusion: if $ D \in \mathcal{D} $ and $ D^{\prime}(\omega) \subset D(\omega) $ for all $\omega\in\Omega$, then $D^{\prime} \in \mathcal{D}$.

\item A random set $(A(\omega))_{\omega\in\Omega}$ is called a \textit{random attractor} if it is
\begin{itemize} 
\item compact for all $\omega\in\Omega$,
\item $\phi$\textit{-invariant}: $\phi(t,\omega,A(\omega)) = A(\vartheta_{t}\omega)$ for all $t\in\mathbb{R}_{\geq0}$, and
\item \textit{pathwise pullback attracting}: for all $D\in\mathcal{D}$
\[
d^{\ast}(\phi(t,\vartheta_{-t}\omega,D(\vartheta_{-t}\omega)), A(\omega)) \to 0
\]
as $t\to\infty$. Here $ d^{\ast} $ denotes the Hausdorff semi-distance on $\mathbb{R}^{d}$.
\end{itemize}
\end{enumerate}
\end{definition}

We define the shift operator $\vartheta_{t}\colon\Omega \to \Omega$ for each $t\in\mathbb{R}$ by $\vartheta_{t}(\omega)_{s} = \omega_{s+t} - \omega_{t} $. It is known that the set $\Omega^{\ast}_{0}$ is shift-invariant: we have $\{ \vartheta_{t} \in \Omega_{0}^{\ast} \} = \Omega_{0}^{\ast}$ for all $t\in\mathbb{R}$ (for the proof, see \cite{gess2011random}). Note that $B(\vartheta_{t}\omega) = \vartheta_{t}B(\omega)$ holds for all $t\in\mathbb{R}$ and $\omega\in\Omega^{\ast}$.

We set $\phi(t,\omega,x) = X^{x,\theta}_{t} (\omega)$, where $ X^{x,\theta}_{t}(\omega) $ denotes the solution of the stochastic differential equation
\[
X_{t}(\omega) = x + \int_{0}^{t}a(X_{s}(\omega),\theta)\,ds + \sigma B_{t}(\omega),\ t\geq 0
\]
on $(\Omega^{\ast},\mathcal{F}^{\ast}_{0},\mathbb{P}^{\ast}_{0})$.

In \cite{garrido2009discretization}, it is proved that 
\begin{itemize}
\item the pair $(\vartheta,\phi)$ defines a continuous RDS, and 
\item this RDS has a random attractor consists of a random element $\{ \bar{X}_{0}^{\theta}(\omega) \}$
\end{itemize}
assuming that the universe $\mathcal{D}$ consists of the \textit{tempered random sets} (see \cite{garrido2009discretization} for detail). 

We set 
\[
\bar{X}^{\theta}_{t}(\omega) = \bar{X}_{0}^{\theta}(\theta_{t}\omega)
\]
for all $t\in\mathbb{R}$. Since $(\mathbb{P}_{0}^{\ast})^{\vartheta_{t}} = \mathbb{P}_{0}^{\ast} $ holds for all $t\in\mathbb{R}$, the process $\bar{X}^{\theta}=(\bar{X}^{\theta}_{t})_{t\in\mathbb{R}}$ is stationary.

\begin{remark}[Some measurability issues]
\begin{asparaenum}
\item The map $\omega \mapsto \bar{X}^{\theta}_{0}(\omega)$ is $\mathcal{F}_{0}^{\ast}$-measurable. This is because $\bar{X}^{\theta}_{0}(\omega)$ can be written as
\[ \bar{X}^{\theta}_{0}(\omega) = \lim_{t\to\infty}\phi(t,\vartheta_{-t}\omega,0).
\]
Note that a (deterministic) one-point set $D=\{p\}$ is a tempered random set.
\item It is clear that $(t,\omega)\mapsto \bar{X}^{\theta}_{t}(\omega)$ is $\mathcal{B}(\mathbb{R})\otimes\mathcal{F}_{0}^{\ast}$-measurable (and hence $\mathcal{B}(\mathbb{R})\otimes\mathcal{F}^{\ast}$-measurable) by Definition \ref{definition1}.
\item The shift operator $\vartheta_{t}$ is also $(\mathcal{F}^{\ast};\mathcal{F}^{\ast})$-measurable for each $t\in\mathbb{R}$. Indeed, $A\in\mathcal{F}^{\ast}$ if and only if there are $A^{\prime},A^{\prime\prime}\in\mathcal{F}^{\ast}_{0}$ such that $A^{\prime} \subset A \subset A^{\prime\prime}$ with $ \mathbb{P}_{0}^{\ast}\{A^{\prime\prime}\setminus A^{\prime}\}=0$. The claim follows because $ \{\vartheta_{t}\in A^{\prime}\} \subset \{ \vartheta_{t} \in A \} \subset \{ \vartheta_{t} \in A^{\prime\prime} \} $ and $\mathbb{P}^{\ast}_{0} \{ \{\vartheta_{t} \in A^{\prime\prime} \} \setminus \{\vartheta_{t} \in A^{\prime}\} \} = \mathbb{P}^{\ast}_{0} \{ A^{\prime\prime}\setminus A^{\prime} \} =0 $. 
\item We can easily verify that $\mathbb{P}^{\ast}\{\vartheta_{t}\in A\} = \mathbb{P}^{\ast}\{A\}$ for all $A\in\mathcal{F}^{\ast}$. Hence $(\mathbb{P}^{\ast})^{\vartheta_{t}}=\mathbb{P}^{\ast}$ holds. In particular, the process $\bar{X}^{\theta}$ is again stationary under the probability measure $\mathbb{P}^{\ast}$.
\end{asparaenum}
\end{remark}

Let us check that the process $\bar{X}^{\theta}$ satisfies the equation (\ref{eq1}). The following proposition gives the proof of Theorem \ref{theorem1} (3).

\begin{proposition}
The stationary process $\bar{X}^{\theta} = (\bar{X}^{\theta}_{t})_{t\in\mathbb{R}}$ satisfies
\begin{align}\label{eq27}
\bar{X}^{\theta}_{t}(\omega) - \bar{X}^{\theta}_{s}(\omega) = \int_{s}^{t} a(\bar{X}^{\theta}_{r}(\omega),\theta)\, dr + \sigma( B_{t}(\omega) - B_{s}(\omega) )
\end{align}
for all $s<t$ and $\omega\in\Omega^{\ast}$.
\end{proposition}

\begin{proof}
By $\phi$-invariance of a random attractor, we have
\begin{align*}
\bar{X}^{\theta}_{t}(\omega) = \phi(t,\omega,\bar{X}_{0}^{\theta}(\omega)) = \bar{X}^{\theta}_{0}(\omega) + \int_{0}^{t} a(X_{s}^{\bar{X}^{\theta}_{0}(\omega),\theta}(\omega),\theta)\, ds + \sigma B_{t}(\omega)
\end{align*}
for all $t\geq 0 $ and $\omega \in \Omega^{\ast}$. Since the (pathwise) solution of the equation (\ref{eq1}) is unique, we have $ X_{t}^{\bar{X}^{\theta}_{0}(\omega),\theta}(\omega) = \bar{X}^{\theta}_{t}(\omega) $ for all $t\geq0 $ and $\omega\in\Omega^{\ast}$, i.e., we obtain 
\begin{align}\label{eq16}
\bar{X}^{\theta}_{t}(\omega) = \bar{X}^{\theta}_{0}(\omega) + \int_{0}^{t} a(\bar{X}^{\theta}_{s}(\omega),\theta)\, ds + \sigma (B_{t}(\omega) - B_{0}(\omega) )
\end{align}
for all $t\geq 0 $ and $\omega \in \Omega^{\ast}$.  
We can replace $\omega$ by $\vartheta_{-t}\omega$. Then we have
\begin{align}\label{eq17}
\bar{X}^{\theta}_{0}(\omega) = \bar{X}^{\theta}_{-t}(\omega) + \int_{-t}^{0}a(\bar{X}^{\theta}_{s}(\omega),\theta)\, ds + \sigma (B_{0}(\omega)-B_{-t}(\omega))
\end{align}
for all $t\geq0$ and $\omega\in\Omega^{\ast}$. Combining (\ref{eq16}) and (\ref{eq17}), we obtain (\ref{eq27}).
\end{proof}

Let us consider applying Malliavin calculus to the stationary solution $\bar{X}^{\theta}$. First we introduce some fractional operators (for detail, see \cite{samko1993fractional}).

\begin{definition}
Let $H$ be in $(0,1/2)$ and $\varphi \in C_{c}^{\infty}(\mathbb{R})$. 
\begin{enumerate}
\item We define a fractional integral of order $1/2-H$ of a function $\varphi$ by 
\[
(I^{1/2-H}_{\pm}\varphi)(x) = \frac{1}{\Gamma(1/2-H)} \int_{0}^{\infty} t^{-H-1/2} \varphi(x\mp t)\, dt
\]
for $x\in\mathbb{R}$. Note that, by Theorem 5.3 of \cite{samko1993fractional}, the operator $I_{-}^{1/2-H}$ can be extended to a bounded operator from $L^{2}(\mathbb{R})$ to $L^{1/H}(\mathbb{R})$.

\item Let $\psi$ be in $I^{1/2-H}_{-}(L^{2}(\mathbb{R}))$. We define a (Marchaud) fractional derivative of order $1/2-H$ of a function $\psi$ by
\[
(\mathbf{D}^{1/2-H}_{\pm}\psi)(x) = L^{2}\mathrm{-}\lim_{\epsilon\to 0}\frac{1/2-H}{\Gamma(H+1/2)}\int_{\epsilon}^{\infty} \frac{\psi(x)-\psi(x\mp t)}{t^{3/2-H}}\, dt
\]
for $x\in\mathbb{R}$. By Theorem 6.1 of \cite{samko1993fractional}, we have \[
\mathbf{D}^{1/2-H}_{-}I_{-}^{1/2-H}f = f
\]
for $f\in L^{2}(\mathbb{R})$.
\end{enumerate}
\end{definition}

Let $\mathcal{H}$ denote the space $I_{-}^{1/2-H}(L^{2}(\mathbb{R}))$ with the inner product 
\[
\langle f,g \rangle_{\mathcal{H}} = e_{H} \langle \mathbf{D}_{-}^{1/2-H}f, \mathbf{D}_{-}^{1/2-H}g \rangle_{L^{2}(\mathbb{R})},
\] 
where $e_{H}=\Gamma(2H+1)\sin(\pi H)$. It is shown in \cite{pipiras2000integration} that the space $\mathcal{H}$ is a Hilbert space. Then the process $B$ defines an isonormal Gaussian process over $\mathcal{H}$ (see also \cite{cheridito2005stochastic}).

Let $L^{2}(\Omega^{\ast};\mathcal{H})$ be the set of $\mathcal{H}$-valued random variables that are square integrable: if $h\in L^{2}(\Omega^{\ast};\mathcal{H})$, then $\mathbb{E}^{\ast}\{ \|h\|_{\mathcal{H}}^{2} \} < \infty$. The subset of $L^{2}(\Omega^{\ast};\mathcal{H})$, which consists of $\mathcal{H}$-valued random variables of the form 
\[
\phi = \sum_{i=1}^{n} Z_{i} \phi_{i},
\]
where $Z_{i}\in L^{2}(\mathbb{P^{\ast}})$ and $\phi_{i} \in I_{-}^{1/2-H}(C_{c}^{\infty}(\mathbb{R}))$ for $i=1,\ldots,n$, is denoted by $\mathcal{G}$. Here $C_{c}^{\infty}(\mathbb{R})$ denotes the set of smooth functions of compact support on $\mathbb{R}$. Note that the set $\mathcal{G}$ is dense in $L^{2}(\Omega^{\ast};\mathcal{H})$.

Now we turn to show the Malliavin differentiability of the stationary solution $\bar{X}^{\theta}$. The following lemma reduces the Malliavin differentiability of $\bar{X}^{\theta}_{t}$ to that of $\bar{X}^{\theta}_{0}$.

\begin{lemma} \label{lemma12}
We set $\tau_{s}\colon\mathcal{H}\to\mathcal{H}$ by $(\tau_{s}f)(t) = f(t-s)$ for $s\in\mathbb{R}$ and $f\in\mathcal{H}$. Suppose that $F$ is in $\mathbb{D}^{1,2}$. Then we have $F\circ\vartheta_{s}\in\mathbb{D}^{1,2}$ and
\begin{align}\label{eq37}
D(F\circ \vartheta_{s}) = \tau_{s}(DF)\circ\vartheta_{s}.
\end{align}

\end{lemma}
\begin{proof}
Let $\tilde{\mathbf{1}}_{(a,b]}$ denote the extended indicator function:
\begin{align*}
\tilde{\mathbf{1}}_{(a,b]} =
\begin{cases}
\mathbf{1}_{(a,b]}, & \text{if }a\leq b \\ 
-\mathbf{1}_{(b,a]}, & \text{if }a > b.
\end{cases}
\end{align*}
Then we have 
\[
D( (B_{b}-B_{a})\circ\vartheta_{s}) = \tilde{\mathbf{1}}_{(a+s, b+s]} = \tau_{s}\tilde{\mathbf{1}}_{(a,b]} 
\]
for any real numbers $a$, $b$ and $s$. Therefore, by linearity of $\tau_{s}$ and $D$, we have $ D(B(\phi)\circ\vartheta_{s}) = \tau_{s}\phi $ for a real number $s$ and a step function $\phi$. Since the set of step functions is dense in $(\mathcal{H},\|\cdot\|_{\mathcal{H}})$ (see Theorem 3.3 of \cite{pipiras2000integration}), for each $\phi\in\mathcal{H} $ there exists a sequence of step functions $(\phi_{n})$ such that $\| \phi - \phi_{n} \|_{\mathcal{H}} \to 0$ as $n\to\infty$. It is clear that $\mathbb{E}^{\ast}\{ ( B(\phi_{n})\circ \vartheta_{s}  - B(\phi)\circ \vartheta_{s} )^{2} \} = \mathbb{E}^{\ast}\{ ( B(\phi_{n})  - B(\phi) )^{2} \} = \| \phi_{n} - \phi \|_{\mathcal{H}}^{2} \to 0 $ as $n\to\infty$. Since $\tau_{s}$ and $\mathbf{D}_{-}^{1/2-H}$ are commutative (see (5.61) in p.111 of \cite{samko1993fractional}) and the Lebesgue measure is translation invariant, we have $ \mathbb{E}^{\ast}\{ \| D(B(\phi_{n})\circ\vartheta_{s}) - D(B(\phi_{m})\circ\vartheta_{s}) \|_{\mathcal{H}}^{2} \} = \| \tau_{s}(\phi_{n}-\phi_{m}) \|_{\mathcal{H}}^{2} = \| \phi_{n}-\phi_{m} \|_{\mathcal{H}}^{2} \to 0 $ as $n,m\to\infty$. Hence we have $ D(B(\phi)\circ\vartheta_{s}) = \tau_{s}\phi $ for all $\phi\in\mathcal{H}$.

Let $\mathcal{S}$ denote the set of the random variables $F$ of the form \[F = f(B(\phi_{1}) , \ldots, B(\phi_{n})) \] for some positive integer $n$ where $\phi_{i}\in \mathcal{H}\ (i=1,\ldots,n)$ and $f$ is an infinitely continuously differentiable function such that all its partial derivatives are of polynomial growth. For $F\in\mathcal{S}$, we have
\begin{align*}
D(F\circ\vartheta_{s}) &= \sum_{i=1}^{m}(\partial_{i}f)(B(\phi^{1})\circ\vartheta_{s},\ldots,B(\phi^{m})\circ\vartheta_{s})\tau_{s}\phi^{i} \\
&= \tau_{s}(DF)\circ\vartheta_{s}
\end{align*}
and hence 
\begin{align}\label{eq28}
\mathbb{E}^{\ast}\{ \| D(F\circ\vartheta_{s}) \|^{2}_{\mathcal{H}} \} = \mathbb{E}^{\ast}\{ \| DF \|^{2}_{\mathcal{H}} \}.
\end{align}

For each $F\in\mathbb{D}^{1,2}$, we can choose $(F_{n})_{n}\subset \mathcal{S}$ such that $ \| F_{n} - F \|_{\mathbb{D}^{1,2}} \to 0 $ as $n\to\infty$. It is clear that $ \mathbb{E}^{\ast}\{ (F_{n}\circ\vartheta_{s} - F\circ\vartheta_{s} )^{2} \} \to 0 $ as $n\to\infty$. We also have $ \mathbb{E}^{\ast}\{ \| D(F_{n}\circ\vartheta_{s}) - D(F_{m}\circ\vartheta_{s}) \|_{\mathcal{H}}^{2} \} = \mathbb{E}^{\ast}\{ \| DF_{n} - DF_{m} \|_{\mathcal{H}}^{2} \}\to 0 $ as $n\to\infty$ by (\ref{eq28}). On the other hand, we have $\mathbb{E}^{\ast}\{ \| D(F_{n}\circ \vartheta_{s}) - \tau_{s}(DF)\circ\vartheta_{s} \|^{2}_{\mathcal{H}} \} = \mathbb{E}^{\ast}\{ \| DF_{n} - DF \|_{\mathcal{H}}^{2} \} \to 0$ as $n\to\infty$. Hence we have $F\circ\vartheta_{s} \in \mathbb{D}^{1,2}$ and (\ref{eq37}).
\end{proof}

The next proposition gives the proof of Theorem \ref{theorem1} (4).

\begin{proposition}\label{proposition3}
It holds that $\bar{X}^{\theta}_{t} \in \cap_{p>0}\mathbb{D}^{1,p}$ for all $t\in\mathbb{R}$ and its Malliavin derivative $D\bar{X}^{\theta}_{t}$ is given by (\ref{ineq2}).
\end{proposition}

Since $\bar{X}_{0}$ is defined by the pathwise limit $\bar{X}^{\theta}_{0}(\omega)= \lim_{t\to\infty}(\bar{X}^{\theta}_{t}\circ\vartheta_{-t})(\omega)$, the Malliavin differentiablity of $\bar{X}_{0}^{\theta}$ is inherited from $X^{\theta}_{t}\circ\vartheta_{-t}$. Therefore, we begin with an analysis of the Malliavin derivative of $X^{\theta}_{t}$.

\begin{lemma}\label{lemma14}
Let $\xi\in\mathbb{D}^{1,2}\cap \left(\bigcap_{p>0} L^{p}(\mathbb{P}^{\ast})\right)$ and $Y^{\theta,\xi,s}$ be a solution of the equation
\[
Y_{t} = \xi + \int_{s}^{t} a(Y_{r},\theta)\,dr + \sigma(B_{t}-B_{s}),\ t\in[s,\infty).
\]
Then $Y^{\theta,\xi,s}_{t}$ is in $\mathbb{D}^{1,2}$ and 
\begin{align} \label{eq43}
DY^{\theta,\xi,s}_{t} = e^{\int_{s}^{t}(\partial_{x}a)(Y_{u}^{\theta,\xi,s},\theta)\,du} D\xi + \sigma\mathbf{1}_{(s,t]}e^{\int_{\cdot}^{t}(\partial_{x}a)(Y_{u}^{\theta,\xi,s},\theta)\, du}
\end{align}
for each $t\in[s,\infty)$.
\end{lemma}

\begin{proof}
Let us consider the Picard approximation $Y_{t}^{0}\equiv \xi$ for $t\in[s,\infty)$ and
\[
Y^{n}_{t} = \xi + \int_{s}^{t}a(Y^{n-1}_{r},\theta)\,dr + \sigma(B_{t}-B_{s}) 
\]
for a positive integer $n$. We have 
\begin{align*}
DY_{t}^{n}  = D\xi + \int_{s}^{t} (\partial_{x}a)(Y_{r}^{n-1},\theta) DY_{r}^{n-1}\, dr + \sigma\mathbf{1}_{(s,t]}
\end{align*}
and
\begin{align*} 
DY_{t}^{n+1} - DY_{t}^{n}  = \int_{s}^{t} D(a(Y_{r}^{n},\theta) - a(Y_{r}^{n-1},\theta))\,dr.
\end{align*}
We consider a continuous version of $\mathcal{H}$-valued process $(t,\omega)\mapsto DY^{n}_{t}(\omega)$. We can bound $ \|DY_{t}^{n}\|_{\mathcal{H}} $ independent of $n$. Indeed, we have
\begin{align*}
\|DY_{t}^{n}\|_{\mathcal{H}} &\leq \|D\xi\|_{\mathcal{H}} + |\sigma|(t-s)^{H} + |\alpha|^{-1} \int_{s}^{t} \|DY_{r_{1}}^{n-1}\|_{\mathcal{H}} \,dr_{1} \\
\nonumber &\leq \|D\xi\|_{\mathcal{H}} + |\sigma|(t-s)^{H} + |\alpha|^{-1} \int_{s}^{t}dr_{1}\ \left( \|D\xi\|_{\mathcal{H}} + |\sigma|(r_{1}-s)^{H} \right) \\ 
\nonumber &\quad +|\alpha|^{-2}\int_{s}^{t}dr_{1}\int_{s}^{r_{1}}dr_{2}\ \|DY_{r_{2}}^{n-2}\|_{\mathcal{H}} \\
\nonumber &\leq \left(\|D\xi\|_{\mathcal{H}} + |\sigma|(t-s)^{H} \right)( 1 + |\alpha|^{-1}(t-s) )\\
\nonumber &\quad +|\alpha|^{-2}\int_{s}^{t}dr_{1}\int_{s}^{r_{1}}dr_{2}\ \|DY_{r_{2}}^{n-2}\|_{\mathcal{H}}.
\end{align*}
By iterating this procedure, we obtain
\[
\|DY_{t}^{n}\|_{\mathcal{H}} \leq \left( \|D\xi\|_{\mathcal{H}} + |\sigma|(t-s)^{H} \right) e^{|\alpha|^{-1}(t-s)}.
\]

Next we bound $\|DY_{t}^{n+1} - DY_{t}^{n}\|_{\mathcal{H}}$. Since $\|DY_{t}^{n}\|_{\mathcal{H}}$ is bounded by a constant independent of $n$, we have
\begin{align*}
& \|D(a(Y_{r}^{n},\theta) - a(Y_{r}^{n-1},\theta))\|_{\mathcal{H}} \\
&= \Biggl\| \int_{0}^{1} (\partial_{x}\partial_{x}a)((1-\epsilon)Y_{r}^{n-1}+\epsilon Y_{r}^{n}, \theta) ((1-\epsilon)DY_{r}^{n-1} + \epsilon DY_{r}^{n} )\, d\epsilon \\
&\quad \times  (Y_{r}^{n}-Y_{r}^{n-1}) + \int_{0}^{1}(\partial_{x}a)((1-\epsilon)Y_{r}^{n-1} + \epsilon Y_{r}^{n} , \theta)\,d\epsilon (DY_{r}^{n}-DY_{r}^{n-1}) \Biggr\|_{\mathcal{H}} \\
& \leq |\alpha|^{-1} \left( \|D\xi\|_{\mathcal{H}} + |\sigma|(t-s)^{H} \right) e^{|\alpha|^{-1}(t-s)} |Y_{r}^{n}-Y_{r}^{n-1}|  \\
&\quad + |\alpha|^{-1} \|DY_{r}^{n} - DY_{r}^{n-1}\|_{\mathcal{H}}.
\end{align*}
Therefore, we obtain 
\begin{align*}
&\|DY_{t}^{n+1}-DY_{t}^{n}\|_{\mathcal{H}}\\
&\leq\int_{s}^{t} \|D(a(Y_{r}^{n},\theta) - a(Y_{r}^{n-1},\theta))\|_{\mathcal{H}}\,dr\\
&\leq |\alpha|^{-1} \left( \|D\xi\|_{\mathcal{H}} + |\sigma|(t-s)^{H} \right) e^{|\alpha|^{-1}(t-s)} \int_{s}^{t}dr_{1}\ |Y_{r_{1}}^{n}-Y_{r_{1}}^{n-1}| \\
&\quad + |\alpha|^{-1} \int_{s}^{t}dr_{1}\ \|DY_{r_{1}}^{n} - DY_{r_{1}}^{n-1}\|_{\mathcal{H}} \\
&\leq |\alpha|^{-1} \left( \|D\xi\|_{\mathcal{H}} + |\sigma|(t-s)^{H} \right) e^{|\alpha|^{-1}(t-s)} \\
&\quad \times\left( \int_{s}^{t}dr_{1}\ |Y_{r_{1}}^{n}-Y_{r_{1}}^{n-1}| + |\alpha|^{-1} \int_{s}^{t}dr_{1}\int_{s}^{r_{1}}dr_{2}\ |Y_{r_{2}}^{n-1}-Y_{r_{2}}^{n-2}| \right)\\
&\quad+ |\alpha|^{-2}\int_{s}^{t}dr_{1}\int_{s}^{r_{1}}dr_{2}\ \| DY_{r_{2}}^{n-1} - DY_{r_{2}}^{n-2} \|_{\mathcal{H}}.
\end{align*}
By iterating this estimate, we have
\begin{align*}
&\|DY_{t}^{n+1}-DY_{t}^{n}\|_{\mathcal{H}}\\
&\leq |\alpha|^{-1} \left( \|D\xi\|_{\mathcal{H}} + |\sigma|(t-s)^{H} \right) e^{|\alpha|^{-1}(t-s)} \\
&\quad \times \sum_{k=0}^{n-1} |\alpha|^{-k} \int_{s}^{t}dr_{1}\int_{s}^{r_{1}}dr_{2}\cdots\int_{s}^{r_{k}}dr_{k+1}\ |Y_{r_{k+1}}^{n-k}-Y_{r_{k+1}}^{n-k-1}| \\
&\quad + |\alpha|^{-n} \int_{s}^{t}dr_{1}\int_{s}^{r_{1}}dr_{2}\cdots\int_{s}^{r_{n-1}}dr_{n}\ \|DY_{r}^{1}-DY_{r}^{0}\|_{\mathcal{H}}\\
&=: \mathtt{I}_{1}^{n}(t) + \mathtt{I}_{2}^{n}(t).
\end{align*}
Since
\[
|Y_{t}^{n+1} - Y_{t}^{n}| \leq (|a(\xi,\theta)|(t-s) + |\sigma||B_{t}-B_{s}|)\frac{(|\alpha|^{-1}(t-s))^{n}}{n!}
\]
holds, the term $\mathtt{I}_{1}^{n}(t)$ is bounded as 
\begin{align*}
\mathtt{I}_{1}^{n}(t) &\leq 
|\alpha|^{-1} \left( \|D\xi\|_{\mathcal{H}} + |\sigma|(t-s)^{H} \right) e^{|\alpha|^{-1}(t-s)} \\
&\quad \times (|a(\xi,\theta)|(t-s) + |\sigma| \sup_{s\leq r\leq t} |B_{t}-B_{s}| ) \\
&\quad \times\sum_{k=0}^{n-1} |\alpha|^{-k} \int_{s}^{t}dr_{1}\int_{s}^{r_{1}}dr_{2}\cdots\int_{s}^{r_{k}}dr_{k+1}\ \frac{(|\alpha|^{-1}(r_{k+1}-s))^{n-k-1}}{(n-k-1)!} \\
&\leq |\alpha|^{-1}\left( \|D\xi\|_{\mathcal{H}} + |\sigma|(t-s)^{H} \right) e^{|\alpha|^{-1}(t-s)} \\
&\quad \times (|a(\xi,\theta)|(t-s) + |\sigma| \sup_{s\leq r\leq t} |B_{r}-B_{s}| ) (t-s) \frac{(|\alpha|^{-1}(t-s))^{n-1}}{(n-1)!}.
\end{align*}
The term $\mathtt{I}_{2}^{n}(t)$ can be bounded as 
\[
\mathtt{I}_{2}^{n}(t) \leq (|\alpha|^{-1}\|D\xi\|_{\mathcal{H}}(t-s) + |\sigma|(t-s)^{H}) \frac{(|\alpha|^{-1}(t-s))^{n}}{n!}.
\]

Therefore we have, for each $T>s$, 
\[
 \sup_{s\leq t \leq T} \| DY^{n}_{t} - DY^{m}_{t} \|_{\mathcal{H}} \to 0
\]
pointwisely and in $L^{2}(\mathbb{P}^{\ast})$ as $n,m\to\infty$. Hence $Y^{\theta,\xi,s}_{t} \in \mathbb{D}^{1,2}$. 

A continuous version of $(DY^{\theta,\xi,s}_{t})_{t\in[s,\infty)}$ satisfies 
\begin{align}\label{eq44}
D^{\phi}Y^{\theta,\xi,s}_{t} = D^{\phi}\xi + \int_{s}^{t} (\partial_{x}a)(Y^{\theta,\xi,s}_{r},\theta) D^{\phi}Y^{\theta,\xi,s}_{r} \,dr + \int_{s}^{t}b(r)\,dr
\end{align}
for all $t\in[s,\infty)$ and $\phi\in I_{-}^{1/2-H}(C_{c}^{\infty}(\mathbb{R}))$, where 
\[
b(r) = \sigma (\mathbf{D}^{1/2-H}_{+}\mathbf{D}^{1/2-H}_{-}\phi)_{r}.
\]
Here we used the integration by parts formula for fractional derivatives (see p.129 of \cite{samko1993fractional}):
\begin{align*}
 \sigma \langle \mathbf{1}_{(s,t]}, \phi \rangle_{\mathcal{H}} &= \int_{\mathbb{R}} (\mathbf{D}_{-}^{1/2-H}\mathbf{1}_{(s,t]})_{s} (\mathbf{D}_{-}^{1/2-H}\phi)_{s}\,ds \\
\nonumber &= \sigma\int_{s}^{t} (\mathbf{D}^{1/2-H}_{+}\mathbf{D}_{-}^{1/2-H}\phi)_{s} ds.
\end{align*}
Solving the equation (\ref{eq44}), we have
\begin{align*}
D^{\phi}Y_{t}^{\theta,\xi,s} = \langle e^{\int_{s}^{t}(\partial_{x}a)(Y_{u}^{\theta,\xi,s},\theta)\,du} D\xi + \sigma\mathbf{1}_{(s,t]}e^{\int_{\cdot}^{t}(\partial_{x}a)(Y_{u}^{\theta,\xi,s},\theta)\, du} , \phi \rangle_{\mathcal{H}}
\end{align*}
for all $\phi\in I_{-}^{1/2-H}(C_{c}^{\infty}(\mathbb{R}))$. Since $I_{-}^{1/2-H}(C_{c}^{\infty}(\mathbb{R}))$ is dense in $\mathcal{H}$, we obtain the identity (\ref{eq43}).
\end{proof}

\begin{lemma}\label{lemma13}
There is a positive constant $C>0$ such that, for all $t\geq0$,
\begin{align} 
\| DX_{t}^{0,\theta} \|_{\mathcal{H}}^{2} \leq C
\end{align}
holds $\mathbb{P}^{\ast}$-a.s.
\end{lemma}
\begin{proof}
Thanks to Lemma \ref{lemma14}, we have $X_{t}^{0,\theta}\in\mathbb{D}^{1,2}$ and 
\[
 DX_{t}^{0,\theta} = \sigma\mathbf{1}_{(0,t]}(\cdot)e^{\int_{\cdot}^{t}(\partial_{x}a)(X_{u}^{0,\theta},\theta)\, du}.
 \]
We set $ \Phi(t,s) = \sigma\mathbf{1}_{(0,t]}(s)e^{\int_{s}^{t}\partial_{x}a(\bar{X}^{\theta}_{v},\theta)dv} $.
Therefore $ \|DX_{t}^{0,\theta}\|_{\mathcal{H}} = \| \Phi(t,\cdot) \|_{\mathcal{H}} $ holds $\mathbb{P}^{\ast}$-a.s., and so that it suffices to show that there is a positive constant $C>0$ such that $\| \Phi(t,\cdot) \|_{\mathcal{H}} \leq C$ holds for all $t\geq0$. 

First we consider the following decomposition:
\begin{align*}
\| \Phi(t,\cdot) \|_{\mathcal{H}}^{2} &= \int_{-\infty}^{0}ds\left| \int_{-s}^{-s+t}d\xi\,\xi^{H-3/2}\Phi(t,s+\xi) \right|^{2} \\
&\quad + \int_{0}^{t}ds\left| \int_{0}^{\infty}d\xi\,\xi^{H-3/2}(\Phi(t,s)-\Phi(t,s+\xi)) \right|^{2} \\
&=: \mathtt{I}_{1}(t) + \mathtt{I}_{2}(t).
\end{align*}
The term $\mathtt{I}_{1}(t)$ can be bounded from above as follows:
\begin{align*}
\mathtt{I}_{1}(t) &= |\sigma|^{2}\int_{-\infty}^{0}ds\left| \int_{0}^{t}d\xi\,(\xi-s)^{H-3/2}e^{\int_{\xi}^{t}\partial_{x}a(X_{v}^{\theta},\theta)dv} \right|^{2} \\
&\leq |\sigma|^{2}\int_{-\infty}^{0}ds\left| \int_{0}^{t}d\xi\,(\xi-s)^{H-3/2}e^{\alpha(t-\xi)} \right|^{2} \\
&= |\sigma|^{2}\int_{0}^{t}d\xi\int_{0}^{t}d\eta\,e^{-\alpha(t-\xi)}e^{-\alpha(t-\eta)} \int_{-\infty}^{0}ds\,(\xi-s)^{H-3/2}(\eta-s)^{H-3/2} \\
&\leq (2-2H)^{-1}|\sigma|^{2}\left| \int_{0}^{t}d\xi\,\xi^{H-1}e^{-\alpha(t-\xi)} \right|^{2}.
\end{align*}
Here the last inequality is due to H\"{o}lder's inequality. In order to obtain an upper bound for $\mathtt{I}_{2}(t)$, we further decompose $\mathtt{I}_{2}(t)$ as follows:
\begin{align*}
\mathtt{I}_{2}(t) &\leq 2 \int_{0}^{t}ds\left| \int_{0}^{-s+t}d\xi\,\xi^{H-3/2}(\Phi(t,s)-\Phi(t,s+\xi)) \right|^{2} \\
&\quad + 2\int_{0}^{t}ds\left|\int_{-s+t}^{\infty} d\xi\,\xi^{H-3/2}\Phi(t,s)\right|^{2} \\
&=: \mathtt{I}_{2,1}(t) + \mathtt{I}_{2,2}(t).
\end{align*}
For the term $\mathtt{I}_{2,1}$, we have
\begin{align*}
\nonumber \mathtt{I}_{2,1}(t) &\leq 4\int_{0}^{t}ds\left| \int_{s}^{(s+1)\wedge t}d\xi\,(\xi-s)^{H-3/2}(\Phi(t,s)-\Phi(t,\xi)) \right|^{2}  \\
\nonumber &\quad + 4\int_{0}^{t}ds\left| \int_{(s+1)\wedge t}^{t}d\xi\,(\xi-s)^{H-3/2}(\Phi(t,s)-\Phi(t,\xi)) \right|^{2} \\
\nonumber &\leq 4 |\sigma\alpha^{-1}|^{2} \int_{0}^{t}ds\left| \int_{s}^{(s+1)\wedge t}d\xi\,(\xi-s)^{H-1/2} e^{-\alpha(t-\xi)} \right|^{2}  \\
\nonumber &\quad + 4|\sigma|^{2}\int_{0}^{t}ds\left| \int_{(s+1)\wedge t}^{t}d\xi\,(\xi-s)^{H-3/2}e^{-\alpha(t-\xi)} \right|^{2} \\
\nonumber &= 4|\sigma\alpha^{-1}|^{2} \int_{0}^{(t-1)\vee0}ds\left| \int_{s}^{s+1}d\xi\,(\xi-s)^{H-1/2} e^{-\alpha(t-\xi)} \right|^{2} \\
\nonumber &\quad+ 4 |\sigma\alpha^{-1}|^{2} \int_{(t-1)\vee0}^{t}ds\left| \int_{s}^{t}d\xi\,(\xi-s)^{H-1/2} e^{-\alpha(t-\xi)} \right|^{2} \\
\nonumber &\quad + 4|\sigma|^{2} \int_{0}^{(t-1)\vee0}ds\left| \int_{(s+1)\wedge t}^{t} d\xi\,(\xi-s)^{H-3/2}e^{-\alpha(t-\xi)} \right|^{2}\\
&=: \mathtt{I}_{2,1,1}(t) + \mathtt{I}_{2,1,2}(t) + \mathtt{I}_{2,1,3}(t)
\end{align*}
If $t\in[0,1]$, then only the term $\mathtt{I}_{2,1,2}(t)$ appears and this is finite. If $t>1$, we have
\begin{align*}
\mathtt{I}_{2,1,1}(t) \leq 4(H+1/2)^{-2}|\sigma\alpha^{-1}|^{2} \int_{0}^{t-1}ds\ e^{-2\alpha((t-1)-s)},
\end{align*}
\begin{align*}
\mathtt{I}_{2,1,2}(t) &\leq 4 |\sigma\alpha^{-1}|^{2} \int_{t-1}^{t}ds\left| \int_{s}^{t}d\xi\ (\xi-s)^{H-1/2} \right|^{2} \\
&= |\sigma\alpha^{-1}|^{2}(H+1/2)^{-2}(2H+2)^{-1}
\end{align*}
and 
\begin{align*}
\mathtt{I}_{2,1,3}(t) &= 4|\sigma|^{2}\int_{0}^{t-1}ds\left| e^{-\alpha(t-s)} \int_{1}^{t-s}d\xi\ \xi^{H-3/2}e^{\alpha\xi} \right|^{2} \\
&\lesssim 4|\sigma|^{2} \int_{0}^{t-1}ds\ (t-s)^{2H-3}.
\end{align*}
Here we used Lemma \ref{lemma9} below. On the other hand, for the term $\mathtt{I}_{2,2}(t)$, we have 
\begin{align*}
\mathtt{I}_{2,2}(t) &= 2 (1/2-H)^{-2} \int_{0}^{t}ds\left| \Phi(t,s) (-s+t)^{H-1/2} \right|^{2} \\
&\leq 2|\sigma|^{2}(1/2-H)^{-2}\int_{0}^{t}ds\ e^{-2\alpha(t-s)}(t-s)^{2H-1} \\
&= 2|\sigma|^{2}(1/2-H)^{-2} \int_{0}^{t}ds\ s^{2H-1}e^{-2\alpha s}.
\end{align*}
All these bounds are finite even if we take the supremum over $t\geq0$.
\end{proof}

\begin{lemma}\label{lemma9}
Let $\alpha>0$ and $\beta>0$ be positive constants. For $x\geq1$, there exists a positve constant $C>0$ which depends only on $\alpha$ and $\beta$ such that
\begin{align}
e^{-\alpha x} \int_{1}^{x} d\xi\ \xi^{-\beta}e^{\alpha\xi} \leq C x^{-\beta}
\end{align}
holds.
\end{lemma}
\begin{proof}
Let us set $ \mathtt{I}_{\alpha,\beta}(x) = e^{-\alpha x} \int_{1}^{x} d\xi\ \xi^{-\beta}e^{\alpha\xi} $. We have
\begin{align*}
\mathtt{I}_{\alpha,\beta}(x) &= e^{-\alpha x}\left( \alpha^{-1}x^{-\beta}e^{\alpha x}-\alpha^{-1}e^{\alpha} +\alpha^{-1}\beta\int_{1}^{x}d\xi\ \xi^{-\beta-1}e^{\alpha\xi} \right)\mathbf{1}_{(2,\infty)}(x)\\
\nonumber &\quad  + e^{-\alpha x}\int_{1}^{x}d\xi\ \xi^{-\beta}e^{\alpha\xi}\mathbf{1}_{[1,2]}(x) \\
&\leq e^{-\alpha x}\bigl( \alpha^{-1}x^{-\beta}e^{\alpha x} + \alpha^{-1}\beta\bigl( \beta^{-1}e^{(\alpha/2)x} + e^{\alpha x}(x/2)^{-\beta} \bigr) \bigr)\mathbf{1}_{(2,\infty)}(x) \\
&\quad + e^{\alpha}\mathbf{1}_{[1,2]}(x)
\end{align*}
This completes the proof.
\end{proof}

\begin{proof}[Proof of Proposition \ref{proposition3}]

First we show that $\bar{X}^{\theta}_{t}\in\mathbb{D}^{1,2}$. By Lemma \ref{lemma12}, it suffices to show that $\bar{X}^{\theta}_{0} \in \mathbb{D}^{1,2}$. Since a random attractor is pathwise pullback attracting, we have 
\[
 | X^{0,\theta}_{t}(\vartheta_{-t}\omega) - \bar{X}^{\theta}_{0}(\omega) | \to 0 
\]
as $n\to\infty$ for all $\omega\in\Omega^{\ast}$. Note that for $p>0$ it holds that $ \mathbb{E}^{\ast}\{ |X_{t}^{0,\theta}(\vartheta_{-t})|^{p} \} = \mathbb{E}^{\ast}\{ |X_{t}^{0,\theta}|^{p} \} \leq c_{p} $ with a positive constant $c_{p} >0$, which is independent of $t\geq 0$. Therefore the family $ (|X_{t}^{0,\theta}\circ\vartheta_{-t}|^{p})_{t\geq0} $ is uniformly integrable for all $p>0$ and in particular $L^{p}(\mathbb{P}^{\ast})$-convergence $ X^{0,\theta}_{t}\circ\vartheta_{-t} \to \bar{X}^{\theta}_{0} $ as $t\to\infty$ holds for all $p>0$. Therefore, in order to prove that $ \bar{X}^{\theta}_{0} \in \mathbb{D}^{1,2} $, it suffices to show
\[
\sup_{t\geq0} \mathbb{E}^{\ast}\{ \|D(X_{t}^{0,\theta}\circ \vartheta_{-t})\|_{\mathcal{H}}^{2} \} < \infty
\]
(see Lemma 1.2.3 of \cite{nualart2006malliavin}). By Lemma \ref{lemma12}, it suffices to prove $\sup_{t\geq0} \mathbb{E}^{\ast}\{ \|DX_{t}^{0,\theta}\|_{\mathcal{H}}^{2} \} < \infty$, and this inequality follows from Lemma \ref{lemma13}.

Next we show that $D\bar{X}^{\theta}_{t}$ is given by (\ref{ineq2}). Let $\phi \in \mathcal{G}$. By (\ref{eq43}), we have 
\begin{align}\label{eq32}
D^{\phi}\bar{X}^{\theta}_{t} = e^{\int_{s}^{t}(\partial_{x}a)(\bar{X}^{\theta}_{r},\theta)\,dr} D^{\phi}\bar{X}^{\theta}_{s} + \sigma \int_{s}^{t} e^{\int_{r}^{t}(\partial_{x}a)(\bar{X}^{\theta}_{u},\theta)\,du} (\mathbf{D}^{1/2-H}_{+}\mathbf{D}^{1/2-H}_{-}\phi)_{r}\,dr
\end{align}
$\mathbb{P}^{\ast}$-a.s., for each $s<0$. By Assumption \ref{assumption1} and Lemma \ref{lemma12}, the first term in (\ref{eq32}) converges to zero in $L^{2}(\mathbb{P}^{\ast})$ as $s\to -\infty$. Moreover, since 
\[
( \mathbf{D}^{1/2-H}_{+}\mathbf{D}^{1/2-H}_{-}\phi )_{r} = 0 
\]
for sufficiently small $r$, the second term in (\ref{eq32}) coincides with 
\[
 \sigma \int_{-\infty}^{t} e^{\int_{r}^{t}(\partial_{x}a)(\bar{X}^{\theta}_{u},\theta)\,du} (\mathbf{D}^{1/2-H}_{+}\mathbf{D}^{1/2-H}_{-}\phi)_{r}\,dr
\]
if $s$ is sufficiently small. Hence we obtain 
\begin{align}\label{eq33}
D^{\phi}\bar{X}^{\theta}_{t} = \sigma \int_{\mathbb{R}} \mathbf{D}^{1/2-H}_{-}(\mathbf{1}_{(-\infty,t]}(\cdot) e^{\int_{\cdot}^{t}(\partial_{x}a)(\bar{X}^{\theta}_{u},\theta)\,du} )_{r} (\mathbf{D}^{1/2-H}_{-}\phi)_{r}\,dr
\end{align}
$\mathbb{P}^{\ast}$-a.s. Let us denote the function $ r \mapsto \sigma \mathbf{1}_{(-\infty,t]}(r) e^{\int_{r}^{t}(\partial_{x}a)(\bar{X}^{\theta}_{u},\theta)\,du}$ by $ r \mapsto \Psi(t,r)$ for simplicity.
Taking the expectation of the both sides in $(\ref{eq33})$, we have $ \langle D\bar{X}^{\theta}_{t}, \phi \rangle_{L^{2}(\Omega;\mathcal{H})} = \langle \Psi(t,\cdot),\phi  \rangle_{L^{2}(\Omega;\mathcal{H})} $. Since the set $\mathcal{G}$ is dense in $L^{2}(\Omega;\mathcal{H})$, we obtain (\ref{ineq2}).

Finally we prove $ \mathbb{E}^{\ast}\{\| D\bar{X}^{\theta}_{t} \|^{p}_{\mathcal{H}}\} < \infty$ for all $p>0$. It suffices to show that there is a deterministic constant $C>0$ such that
\begin{align}\label{ineq21}
\|\Psi(0,\cdot)\|_{\mathcal{H}}^{2} \leq C
\end{align}
holds. A straightforward calculation yields 
\begin{align*}
(\mathbf{D}^{1/2-H}_{-} \Psi(0,\cdot))_{t} &= \biggl(\sigma\int_{0}^{-t}d\xi\ \xi^{H-3/2}\left(e^{\int_{t}^{0} (\partial_{x}a)(\bar{X}^{\theta}_{r},\theta)\,dr} - e^{\int_{t+\xi}^{0} (\partial_{x}a)(\bar{X}^{\theta}_{r},\theta)\,dr} \right) \\
&\quad + \sigma\int_{-t}^{\infty} d\xi\ \xi^{H-3/2}e^{\int_{t}^{0} (\partial_{x}a)(\bar{X}^{\theta}_{r},\theta)\,dr}\biggr) \mathbf{1}_{(-\infty,0]}(t)\\
&=:\mathtt{I}_{1}(t) + \mathtt{I}_{2}(t).
\end{align*}
By a simple calculation, we have 
\[
|\mathtt{I}_{1}(t)|\mathbf{1}_{[-1,0]}(t) \leq |\alpha^{-1}\sigma|(H+1/2)^{-1}(-t)^{H+1/2}\mathbf{1}_{[-1,0]}(t),
\]
\begin{align*}
&|\mathtt{I}_{1}(t)|\mathbf{1}_{(-\infty,-1)}(t) \\
&\leq \Biggl| \sigma \int_{0}^{1}d\xi\ \xi^{H-3/2}\left( e^{\int_{t}^{0} (\partial_{x}a)(\bar{X}^{\theta}_{r},\theta)\,dr} - e^{\int_{t+\xi}^{0} (\partial_{x}a)(\bar{X}^{\theta}_{r},\theta)\,dr} \right) \\
&\quad +\sigma\int_{1}^{-t}d\xi\ \xi^{H-3/2}\left( e^{\int_{t}^{0} (\partial_{x}a)(\bar{X}^{\theta}_{r},\theta)\,dr} - e^{\int_{t+\xi}^{0} (\partial_{x}a)(\bar{X}^{\theta}_{r},\theta)\,dr} \right) \Biggr|\mathbf{1}_{(-\infty,-1)}(t) \\
&\leq \left( |\alpha^{-1}\sigma|e^{\alpha t}\int_{0}^{1}d\xi\ \xi^{H-1/2}e^{\alpha\xi} + |\sigma|e^{\alpha t}\int_{1}^{-t}d\xi\ \xi^{H-3/2}e^{\alpha\xi} \right) \mathbf{1}_{(-\infty,-1)}(t)
\end{align*}
and
\[
|\mathtt{I}_{2}(t)| \leq |\sigma|(1/2-H)e^{\alpha t} (-t)^{H-1/2}\mathbf{1}_{(-\infty,0]}(t).
\]

After all, we obtain the inequality
\begin{align}\label{ineq22}
|(\mathbf{D}^{1/2-H}_{-} \Psi(0,\cdot))_{t}| \lesssim ((-t)^{H-1/2}\mathbf{1}_{[-1,0]}(t) + (-t)^{H-3/2}\mathbf{1}_{(-\infty,-1)}(t)).
\end{align}
This upper bound is in $L^{1}(\mathbb{R})\cap L^{2}(\mathbb{R})$.
\end{proof}

\section{Proof of the first part of Theorem \ref{theorem2}}

First we introduce some transformations related to fBM.

\begin{definition}
\begin{asparaenum}
\item[(1)] Let $ K_{H}^{\ast}\colon I_{T-}^{1/2-H}(L^{2}[0,T]) \to L^{2}[0,T] $ be a map defined by
\begin{align*}
(K_{H}^{\ast}h)_{s} &= d_{H}s^{1/2-H}D_{T-}^{1/2-H}(\cdot^{H-1/2}h_{\cdot})_{s}\\ 
&= \frac{d_{H}s^{1/2-H}}{\Gamma(1-(1/2-H))} \Biggl( \frac{s^{H-1/2}h_{s}}{(T-s)^{1/2-H}} \\
&\quad + (1/2-H)\int_{s}^{T}\frac{ s^{H-1/2}h_{s} -  r^{H-1/2}h_{r} }{ (r-s)^{3/2-H} }\, dr \Biggr).
\end{align*}
Note that the inverse $K_{H}^{\ast,-1} : L^{2}[0,T] \to I_{T-}^{1/2-H}(L^{2}[0,T])$ of $K_{H}^{\ast}$ is well-defined and given by 
\begin{align*}
(K_{H}^{\ast,-1}g)_{s} &= d_{H}^{-1}s^{1/2-H} I_{T-}^{1/2-H}( \cdot^{H-1/2} g_{\cdot} )_{s}\\
&= \bar{d}_{H}^{-1} s^{1/2-H}\int_{s}^{T}r^{H-1/2}(r-s)^{-1/2-H}g_{r}\,dr.
\end{align*}
For properties of the operators $K_{H}^{\ast}$ and $K_{H}^{\ast,-1}$, we refer to \cite{alos2001stochastic} and \cite{nualart2006malliavin}.
\item[(2)] We define a Volterra kernel $K_{H}(t,s)$ by
\[
K_{H}(t,s) = (K^{\ast}_{H}\mathbf{1}_{[0,t]})_{s}.
\] 
Here the symbol $\mathbf{1}_{A}(x)$ denotes an indicator function which is $1$ if $x\in A$ and $0$ otherwise. For explicit expressions of the kernel $K_{H}(t,s)$, see \cite{decreusefond1999stochastic}, \cite{alos2001stochastic} and \cite{nualart2006malliavin}.
\item[(3)] Let $K_{H}\colon L^{2}[0,T] \to I_{0+}^{1/2-H}(L^{2}[0,T])$ be such that 
\[
(K_{H}g)_{t} = \int_{0}^{T} K_{H}(t,s) g_{s}\, ds
\]
for $g \in L^{2}[0,T]$. Then $K_{H}$ defines an isomorphism between $L^{2}[0,T]$ and $I_{0+}^{1/2-H}(L^{2}[0,T])$. The operator $K_{H}$ can be expressed as 
\[
(K_{H}g)_{t} = d_{H} I_{0+}^{2H}(\cdot^{1/2-H}I_{0+}^{1/2-H}(\cdot^{H-1/2}g_{\cdot})_{\cdot})_{t},
\]
and hence its inverse $K_{H}^{-1} \colon I_{0+}^{1/2-H}(L^{2}[0,T]) \to L^{2}[0,T] $
\[
(K_{H}^{-1}h)_{t} = d_{H}^{-1} t^{1/2-H} D_{0+}^{1/2-H}(\cdot^{H-1/2}(D_{0+}^{2H}h_{\cdot})_{\cdot})_{t}
\]
for $h\in I_{0+}^{1/2-H}(L^{2}[0,T])$. For properties of the operators $K_{H}$ and $K_{H}^{-1}$, we refer to \cite{decreusefond1999stochastic} and \cite{nualart2002regularization}.

\end{asparaenum}
\end{definition}

\begin{remark}
If $h$ is absolutely continuous, then 
\begin{align}\label{eq25}
\nonumber (K_{H}^{-1}h)_{t} &= d_{H}^{-1}t^{H-1/2}I_{0+}^{1/2-H}\left(\cdot^{1/2-H}\dot{h}\right)_{t} \\
&= \bar{d}_{H}^{-1}t^{H-1/2}\int_{0}^{t}s^{1/2-H}(t-s)^{-1/2-H}\dot{h}_{s}\,ds.
\end{align}
holds (see \cite{nualart2002regularization}).
\end{remark}

To derive the likelihood ratio formula (\ref{eq9}), we follow the approach of \cite{tudor2007statistical} (see also \cite{nualart2002regularization}).

By the definition of $B^{\theta}$, we have 
\begin{align}\label{eq27}
\sigma^{-1}(\pi_{t}-\pi_{0}) = \int_{0}^{t} \sigma^{-1} a(\pi_{s},\theta)\,ds + B^{\theta}_{t}
\end{align}
for each $t\in[0,T]$. As in \cite{tudor2007statistical}, we define $(\mathcal{B}_{t})$-Brownian motion $W^{\theta}=(W^{\theta}_{t})$ by the Wiener integral
\begin{align}\label{eq26}
W_{t}^{\theta} = \int_{0}^{t} (K_{H}^{\ast,-1}\mathbf{1}_{[0,t]})_{s} \, dB_{s}^{\theta}.
\end{align}
Then we obtain
\begin{align}\label{eq5}
\sigma^{-1}(\pi_{t}-\pi_{0}) = \int_{0}^{t}K_{H}(t,s)\,dZ^{\theta}_{s},
\end{align}
where
\[
Z_{t}^{\theta} = W_{t}^{\theta} + \int_{0}^{t} K_{H}^{-1}\left( \int_{0}^{\cdot}\sigma^{-1}a(\pi_{r},\theta)\,dr\right)_{s}\,ds.
\]
Note that we have 
\[
K_{H}^{-1}\left( \int_{0}^{\cdot}\sigma^{-1}a(\pi_{r},\theta)\,dr\right)_{s} = \beta_{s}(\theta)
\]
by the equation (\ref{eq25}).

Hence the process $Z^{\theta}$ is a Brownian motion with drift under $\mu_{\theta}^{T}$. Now we apply the Girsanov theorem. We can check the Novikov condition is satisfied.

\begin{lemma}
Suppose that Assumption \ref{assumption1} holds. Then the Novikov condition holds:
\begin{align}
\mathbb{E}^{\mu_{\theta}^{(T)}} \left\{ \exp\left( \frac{1}{2}\int_{0}^{T} \beta_{t}^{2}(\theta)\,dt \right) \right\} < +\infty.
\end{align}
\end{lemma}
\begin{proof}
See the proof of Proposition 1 of \cite{tudor2007statistical}.
\end{proof}

Therefore the process $Z^{\theta}$ is a Brownian motion under the probability measure $\nu_{\theta}^{T}$ defined by $ d\nu_{\theta}^{(T)} = z_{T}^{\theta} d\mu_{\theta}^{(T)}$, where 
\[
z_{T}^{\theta} = \exp\left( -\int_{0}^{T} \beta_{t}(\theta)\ dW_{t}^{\theta} - \frac{1}{2}\int_{0}^{T} \beta_{t}^{2}(\theta)\,dt  \right).
\]
We can easily show that $ \mu_{\theta}^{T}\{ z_{T}^{\theta} = 0 \}=0 $, and therefore it holds that $ \mu_{\theta}^{T} \ll \nu_{\theta}^{T} $ (absolutely continous) and 
\begin{align}\label{eq4}
\frac{d\mu_{\theta}^{T}}{d\nu_{\theta}^{T}} = (z_{T}^{\theta})^{-1}.
\end{align}

The probability measure $ \nu_{\theta}^{T} $ actually is independent of $\theta\in\Theta$ by (\ref{eq5}). Therefore we have
\begin{align}\label{eq8}
\frac{d\mu_{\theta^{\prime}}^{T}}{d\mu_{\theta}^{T}} &= \exp\biggl( \int_{0}^{T} \beta_{t}(\theta^{\prime})\, dW_{t}^{\theta^{\prime}} - \int_{0}^{T} \beta_{t}(\theta) \, dW^{\theta}_{t} + \frac{1}{2} \int_{0}^{T}  (\beta_{t}(\theta^{\prime})^{2} - \beta_{t}(\theta)^{2})\,dt \biggr)
\end{align}
under the measure $\mu_{\theta}^{T}$. By (\ref{eq26}) and (\ref{eq27}), we have
\[
W^{\theta^{\prime}}_{t} =  W_{t}^{\theta} + \int_{0}^{t}\sigma^{-1}(a(\pi_{s},\theta)-a(\pi_{s},\theta^{\prime}))(K_{H}^{\ast,-1}\mathbf{1}_{[0,t]})_{s}\,ds.
\]
and
\begin{align*}
&\int_{0}^{t}\sigma^{-1}(a(\pi_{s},\theta)-a(\pi_{s},\theta^{\prime}))(K_{H}^{\ast,-1}\mathbf{1}_{[0,t]})_{s}\,ds \\
&= \bar{d}_{H}^{-1}\sigma^{-1}\int_{0}^{t}ds\,s^{1/2-H}\int_{s}^{t}dr\,(a(\pi_{s},\theta)-a(\pi_{s},\theta^{\prime}))r^{H-1/2}(r-s)^{-H-1/2}\\
&= \bar{d}_{H}^{-1}\sigma^{-1}\int_{0}^{t}dr\,r^{H-1/2}\int_{0}^{r}ds\,(a(\pi_{s},\theta)-a(\pi_{s},\theta^{\prime}))s^{1/2-H}(r-s)^{-H-1/2}\\
&= \int_{0}^{t}(\beta_{r}(\theta)-\beta_{r}(\theta^{\prime}))\,dr.
\end{align*}
Hence
\begin{align}\label{eq7}
W_{t}^{\theta^{\prime}} = \int_{0}^{t}(\beta_{s}(\theta) - \beta_{s}(\theta^{\prime}) )\,ds + W_{t}^{\theta}
\end{align}
holds.
Plugging (\ref{eq7}) into (\ref{eq8}), we obtain the formula (\ref{eq9}). This completes the proof.

\section{Local asymptotic structure of the likelihood ratio process}

The aim of this section is to prove the second part of Theorem \ref{theorem2}. Before proceeding to the proof, we recall the \textit{martingale central limit theorem}.

Let $(\Omega,\mathcal{F},\mathbb{P})$ be a complete probability space and $(\mathcal{F}_{t})_{t\in[0,T]}$ be a filtration. The class of $(\mathcal{F}_{t})$-progressively measurable process $h$ satisfying
\[
\mathbb{P}\left\{ \int_{0}^{T}h^{2}_{t}\,dt < \infty \right\} = 1
\]
is denoted by $\mathcal{M}_{T}$. We assume that there is a $d_{2}$-dimensional Brownian motion $ W = (W^{1},\ldots,W^{d_{2}}) $ on $(\Omega,\mathcal{F},\mathbb{P})$, and the random processes 
\[
 (h^{T,(i,j)}(\theta))_{i=1,\ldots,d_{1}, j=1,\ldots,d_{2}}
\]
are in $ \mathcal{M}_{T} $ for each $T>0$ and $\theta\in\Theta$. We define
\[
\mathcal{I}_{T}(\theta) = ( \mathcal{I}^{1}_{T}(\theta),\ldots, \mathcal{I}_{T}^{d_{1}}(\theta) )^{\star},
\]
where
\[
\mathcal{I}_{T}^{i}(\theta) = \sum_{j=1}^{d_{2}} \int_{0}^{T} h^{T,(i,j)}_{t}(\theta)\, dW_{t}^{j}.
\]
The next result is taken from \cite{kutoyants2004statistical}.

\begin{theorem}\label{theorem3}
Suppose that there exists a (nonrandom) positive definite matrix $I(\theta)=(I_{i,j}(\theta))_{i,j=1,\ldots,d_{1}}$ such that the following convergence takes place:
\begin{align*}
\sum_{l=1}^{d_{1}} \int_{0}^{T} h^{T,(i,l)}_{t}(\theta)h^{T,(j,l)}_{t}(\theta)\,dt \to^{p} I_{i,j}(\theta). 
\end{align*}
Then it holds that
\begin{align*}
\mathcal{I}_{T}(\theta) \to^{d} \mathcal{N}_{d_{1}}(0,I(\theta)).
\end{align*}
Here these limits are taken with respect to the measure $\mathbb{P}$. 
\end{theorem}


\subsection{Proof of the second part of Theorem \ref{theorem2}}

For each $u\in\mathbb{R}^{m}$, the likelihood ratio process $ (d\mu^{T}_{\theta+\varphi_{T}(\theta)u} / d\mu^{T}_{\theta}) $ is denoted by $ Z_{\theta}^{T}(u) $. Let us define
\[
R_{t}^{T}(\theta,u) = u^{\star} \varphi_{T}(\theta) \int_{0}^{1}  ( \partial_{\theta}\beta_{t}(\theta+\epsilon \varphi_{T}(\theta)u) - \partial_{\theta}\beta_{t}(\theta) )\,d\epsilon,
\]
then we have 
\begin{align*}
\log Z_{\theta}^{T}(u) &= u^{\star}\varphi_{T}(\theta)\int_{0}^{T}\partial_{\theta}\beta_{t}(\theta)\ dW_{t}^{\theta} - \frac{1}{2}u^{\star} I(\theta) u \\
\nonumber& \quad - \frac{1}{2} u^{\star} \left(\varphi_{T}(\theta) \int_{0}^{T} \partial_{\theta}\beta_{t}(\theta)( \partial_{\theta}\beta_{t}(\theta) )^{\star}\,dt\ \varphi_{T}(\theta) - I(\theta)\right) u \\
\nonumber&\quad+ \int_{0}^{T}R_{t}^{T}(\theta,u)\, dW_{t}^{\theta} 
+ \int_{0}^{T} u^{\star}\varphi_{T}(\theta)\partial_{\theta}\beta_{t}(\theta)R_{t}^{T}(\theta,u)\,dt \\
\nonumber &\quad -\frac{1}{2} \int_{0}^{T} (R_{t}^{T}(\theta,u))^{2}\,dt.
\end{align*}
The first two terms are said to be the \textit{principal part}, and the last four terms the \textit{negligible part}.

First we identify the limit of the principal part. Let us rewrite the conditions of Theorem \ref{theorem3} in terms of our setting. We can choose the probability space $ (\Omega^{\ast},\mathcal{F}^{\ast},\mathbb{P}^{\ast}) $ from Theorem \ref{theorem1} as underlying probability space. We consider the filtration generated by the fBM $B$, which is denoted by $ (\mathcal{F}_{t}^{\ast})_{t\geq0} $. A Brownian motion $W$ in consideration is defined by $  W = W^{\theta}(X^{\theta}) $, and hence $ d_{2} = 1 $. The $ m $-dimensional process 
\[ 
(\partial_{\theta}\beta(\theta))(X^{\theta}) = ( (\partial_{\theta_{1}}\beta(\theta))(X^{\theta}), \ldots, (\partial_{\theta_{m}}\beta(\theta))(X^{\theta}) )^{\star} 
\]
 corresponds to the random processes $ (h^{T,(i,j)}(\theta))_{i=1,\ldots,d_{1}, j=1,\ldots,d_{2}} $ with $d_{1}=1$ and $d_{2}=m$. 






We set
\[
\kappa(i,\theta) = 
\begin{cases}
-1/2 & \text{if}\ i=1,\ldots,m_{0}(\theta) \\
-(1-H) & \text{if}\ i=m_{0}(\theta)+1,\ldots,m.
\end{cases}
\]

\begin{proposition}\label{lemma8}
Let $\theta \in \Theta$. It holds that for any $\epsilon>0$
\begin{align*}
 \mathbb{P}^{\ast} \left\{ \left| T^{\kappa(i,\theta) + \kappa(j,\theta)}\int_{0}^{T} (\partial_{\theta_{i}}\beta_{t}(\theta))(X^{\theta}) (\partial_{\theta_{j}}\beta_{t}(\theta))(X^{\theta})\,dt - I_{i,j}(\theta)  \right| > \epsilon \right\} \to 0
\end{align*}
as $T\to\infty$ for each $i,j=1\ldots,m$. In particular, We have
\begin{align*}
\varphi_{T}(\theta)\int_{0}^{T}\partial_{\theta}\beta_{t}(\theta)\, dW_{t}^{\theta} \to^{d} \mathcal{N}_{m}(0, I(\theta))
\end{align*}
as $T\to\infty$. Here the limit is taken with respect to $\mu_{\theta}^{T}$. Recall that the Fisher information matrix $I(\theta)=(I_{i,j}(\theta))_{i,j=1,\ldots,m}$ is defined in Theorem \ref{theorem2}.
\end{proposition}

\begin{proof}
We set 
\begin{align}\label{eq20}
I_{i,j}^{T}(\theta)(x) = T^{\kappa(i,\theta) + \kappa(j,\theta)}\int_{0}^{T} (\partial_{\theta_{i}}\beta_{t}(\theta))(x) (\partial_{\theta_{j}}\beta_{t}(\theta))(x)\,dt
\end{align}
for $ x\in C[0,T] $.




\noindent \textbf{Step 1.}
We approximate $I_{i,j}^{T}(\theta)(X^{\theta})$ by $I_{i,j}^{T}(\theta)(\bar{X}^{\theta}|_{[0,T]})$. For simplicity, we omit the restriction $|_{[0,T]}$ in the following. Let $ \Delta_{t}^{(i)}(\theta) $ denote $ \partial_{\theta_{i}}\beta_{t}(\theta)(X^{\theta}) - \partial_{\theta_{i}}\beta_{t}(\theta)(\bar{X}^{\theta}) $. Then we have
\begin{align*}
I_{i,j}^{T}(\theta)(X^{\theta}) &= I_{i,j}^{T}(\theta)(\bar{X}^{\theta}) + T^{\kappa(i,\theta) + \kappa(j,\theta)}\int_{0}^{T} \Delta_{t}^{(i)}(\theta) (\partial_{\theta_{j}}\beta_{t}(\theta))(\bar{X}^{\theta}) dt \\ &\quad+ T^{\kappa(i,\theta) + \kappa(j,\theta)}\int_{0}^{T} \Delta_{t}^{(j)}(\theta) (\partial_{\theta_{i}}\beta_{t}(\theta))(\bar{X}^{\theta}) dt \\
&\quad+ T^{\kappa(i,\theta) + \kappa(j,\theta)}\int_{0}^{T} \Delta_{t}^{(i)}(\theta) \Delta_{t}^{(j)}(\theta) dt.
\end{align*}

%
For $t>0$, the quantity $ \Delta_{t}^{(i)}(\theta) $ can be written as
\begin{align*}
\Delta_{t}^{(i)}(\theta) &= \bar{d}_{H}^{-1} \sigma^{-1} t^{H-1/2} \int_{0}^{t}ds\int_{0}^{1}d\epsilon\ (t-s)^{-1/2-H}s^{1/2-H} \\
&\quad\times (\partial_{x}\partial_{\theta_{i}}a)( \bar{X}^{\theta}_{s} + \epsilon (X^{\theta}_{s} - \bar{X}^{\theta}_{s} ) ) (X^{\theta}_{s} - \bar{X}^{\theta}_{s} ).
\end{align*}
Since $|X^{\theta}_{s} - \bar{X}^{\theta}_{s}|\leq |X^{\theta}_{0} - \bar{X}^{\theta}_{0}|e^{-\alpha s}$ holds under Assumption \ref{assumption1} (for a proof, see \cite{garrido2009discretization} or \cite{neuenkirch2014least}), we have
\begin{align*}
|\Delta_{t}^{(i)}(\theta)| &\lesssim |X_{0}^{\theta}-\bar{X}_{0}^{\theta}| t^{H-1/2}\int_{0}^{t}ds\ (t-s)^{-1/2-H}s^{1/2-H}e^{-\alpha s} \\
&=  |X_{0}^{\theta}-\bar{X}_{0}^{\theta}| \int_{0}^{t}ds\ s^{-1/2-H}\left(1-\frac{s}{t}\right)^{1/2-H}e^{-\alpha(t-s)} \\
&\leq |X_{0}^{\theta}-\bar{X}_{0}^{\theta}| \int_{0}^{t}ds\ s^{-1/2-H}e^{-\alpha(t-s)}.
\end{align*}
In particular, we obtain 
\begin{align*}
\mathbb{E}^{\ast} \int_{0}^{T} \Delta_{t}^{(i)}(\theta)^{2} dt \leq C
\end{align*}
for some constant $C>0$ that is independent of $T>0$.

On the other hand, as we shall see in (\ref{limit6}) below, it holds that 
\[
\mathbb{E}^{\ast} \{I_{i,j}^{T}(\theta)(\bar{X}^{\theta})\} \to I_{i,j}(\theta)
\]
 as $T\to\infty$ for all $i,j=1,\ldots,m$. Hence it holds that
\begin{align}\label{limit2}
\mathbb{P}^{\ast}\{ | I_{i,j}^{T}(\theta)(X^{\theta}) - I_{i,j}^{T}(\theta)(\bar{X}^{\theta}) | \}\to 0
\end{align}
as $T\to\infty$ for all $i,j=1,\ldots,m$.

\noindent \textbf{Step 2.} 
Now we calculate the limit of $I_{i,j}^{T}(\theta)(\bar{X}^{\theta})$ when $T\to\infty$. In the following, we show that the limit
\begin{align}\label{limit6}
\lim_{T\to\infty} \mathbb{E}^{\ast} \left\{ \left| I_{i,j}^{T}(\theta)(\bar{X}^{\theta}) - I_{i,j}(\theta)  \right| \right\} = 0
\end{align}
holds.





First we decompose $ \partial_{\theta_{i}}\beta_{t}(\theta)(\bar{X}^{\theta}) $ into three terms:
\begin{align*}
&\partial_{\theta_{i}}\beta_{t}(\theta)(\bar{X}^{\theta}) \\
&= \sigma^{-1}\bar{d}_{H}^{-1} t^{H-1/2}\int_{0}^{t} (t-r)^{-H-1/2}r^{1/2-H} (\partial_{\theta_{i}} a)(\bar{X}_{r}^{\theta} , \theta)\, dr \\
\nonumber &= \sigma^{-1}\bar{d}_{H}^{-1}t^{H-1/2} \mathbb{E}^{\ast}\{ (\partial_{\theta_{i}} a)(\bar{X}^{\theta}_{0},\theta) \} \int_{0}^{t} (t-r)^{-H-1/2}r^{1/2-H}\, dr\\
\nonumber & \quad + \sigma^{-1}\bar{d}_{H}^{-1} t^{H-1/2}\int_{0}^{t} (t-r)^{-H-1/2}r^{1/2-H} \\
\nonumber&\quad\times \left( (\partial_{\theta_{i}} a)(\bar{X}_{r}^{\theta} , \theta) - \mathbb{E}^{\ast}\{ (\partial_{\theta_{i}} a)(\bar{X}^{\theta}_{0},\theta) \} \right) \, dr \\
\nonumber &= \sigma^{-1}\bar{d}_{H}^{-1}B(-H+1/2,-H+3/2)t^{1/2-H}\mathbb{E}^{\ast}\{ (\partial_{\theta_{i}} a)(\bar{X}^{\theta}_{0},\theta) \}  \\
\nonumber &\quad + \sigma^{-1}\bar{d}_{H}^{-1} \int_{0}^{t} r^{-1/2-H} \left( (\partial_{\theta_{i}} a)(\bar{X}_{t-r}^{\theta} , \theta) - \mathbb{E}^{\ast}\{ (\partial_{\theta_{i}} a)(\bar{X}^{\theta}_{0},\theta) \} \right) \, dr \\
\nonumber &\quad +\sigma^{-1}\bar{d}_{H}^{-1} \int_{0}^{t} r^{-1/2-H} \left\{ \left( 1- \frac{r}{t} \right)^{1/2-H} -1 \right\} \\
\nonumber& \quad\times \left( (\partial_{\theta_{i}} a)(\bar{X}_{t-r}^{\theta} , \theta) - \mathbb{E}^{\ast}\{ (\partial_{\theta_{i}} a)(\bar{X}^{\theta}_{0},\theta) \} \right) \, dr \\
\nonumber &=: \gamma_{t}^{1,i}(\theta) + \gamma_{t}^{2,i}(\theta) + \gamma_{t}^{3,i}(\theta).
\end{align*}

The next lemma plays a crucial role to prove (\ref{limit6}).

\begin{lemma} \label{corollary1}
We set 
\begin{align*}
&c_{i,j}(t) \\
\nonumber&= \mathbb{E}^{\ast}\{ ( \partial_{\theta_{i}}a(\bar{X}^{\theta}_{t},\theta) - \mathbb{E}^{\ast}\{ \partial_{\theta_{i}}a(\bar{X}^{\theta}_{0},\theta) \} )( \partial_{\theta_{j}}a(\bar{X}^{\theta}_{0},\theta) - \mathbb{E}^{\ast}\{ \partial_{\theta_{j}}a(\bar{X}^{\theta}_{0},\theta) \} ) \}.
\end{align*}
for $t\geq0$ and $i,j=1,\ldots,d$. 
\begin{enumerate}
\item We have
\begin{align}\label{ineq14}
\int_{0}^{\infty}dr\int_{0}^{\infty}du\ r^{-H-1/2}u^{-H-1/2}|c_{i,j}(|r-u|)| < \infty
\end{align}
for all $i,j=1,\ldots,m$.
\item It holds that
\begin{align}\label{eq22}
\begin{split}
&\lim_{T\to\infty}\mathbb{E}^{\ast}\left\{T^{-1} \int_{0}^{T}\gamma_{t}^{2,i}(\theta)\gamma_{t}^{2,j}(\theta)\,dt \right\} \\
&= \sigma^{-2}\bar{d}_{H}^{-2}\int_{0}^{\infty}dr\int_{0}^{\infty}du\ r^{-H-1/2}u^{-H-1/2}c_{i,j}(|r-u|)
\end{split}
\end{align}
for all $i,j=1,\ldots,m$.
\item We obtain
\begin{align}\label{eq23}
\lim_{T\to\infty}\mathbb{E}^{\ast}\left\{T^{-1} \int_{0}^{T}\gamma_{t}^{3,i}(\theta)^{2}\,dt \right\} = 0
\end{align}
for all $i=1,\ldots,m$.
\end{enumerate}
\end{lemma}

We postpone the proof of Lemma \ref{corollary1} until Section 4.2. In the following, we consider the following three cases:
\begin{itemize}
\item[(a)] $i=1,\ldots,m_{0}(\theta)$ and $j=1,\ldots,m_{0}(\theta)$,
\item[(b)] $i=1,\ldots,m_{0}(\theta)$ and $j=m_{0}(\theta)+1,\ldots,m$ (or $j=1,\ldots,m_{0}(\theta)$ and $i=m_{0}(\theta)+1,\ldots,m$), and
\item[(c)] $i=m_{0}(\theta)+1,\ldots,m$ and $j=m_{0}(\theta)+1,\ldots,m$.
\end{itemize}

\noindent \textbf{Case (a)}. Let us assume that $i=1,\ldots,m_{0}(\theta)$ and $j=1,\ldots,m_{0}(\theta)$. Note that $\mathbb{E}^{\ast}\{ (\partial_{\theta_{i}} a)(\bar{X}^{\theta}_{0},\theta) \}=0$ and $\kappa(i,\theta)=\kappa(j,\theta)=-1/2$ in this case, and so that $I_{i,j}^{T}(\theta)(\bar{X}^{\theta})$ becomes
\[
I_{i,j}^{T}(\theta)(\bar{X}^{\theta}) = T^{-1} \int_{0}^{T} (\gamma_{t}^{2,i}(\theta) + \gamma_{t}^{3,i}(\theta) ) (\gamma_{t}^{2,j}(\theta) + \gamma_{t}^{3,j}(\theta) )\,dt.
\]
Hence 
\begin{align}\label{ineq9}
\begin{split}
& \mathbb{E}\left\{ \left|I_{i,j}^{T}(\theta)(\bar{X}^{\theta}) - T^{-1} \int_{0}^{T}\gamma_{t}^{2,i}(\theta)\gamma_{t}^{2,j}(\theta)\,dt \right|\right\} \\
&\leq \sum_{\substack{(p,q)\in\{2,3\}^{2} \\ (p,q)\neq(2,2) }} \left(T^{-1} \mathbb{E}^{\ast}\int_{0}^{T} \gamma_{t}^{p,i}(\theta)^{2}\,dt\right)^{1/2} \left(T^{-1}\mathbb{E}^{\ast}\int_{0}^{T} \gamma_{t}^{q,j}(\theta)^{2}\,dt\right)^{1/2}.
\end{split}
\end{align}
Let us set $J_{i,j}^{(a),T}(\theta) = T^{-1} \int_{0}^{T}\gamma_{t}^{2,i}(\theta)\gamma_{t}^{2,j}(\theta)\,dt$. By Lemma \ref{corollary1}, the right hand side of (\ref{ineq9}) tends to $0$ as $T\to\infty$. Furthermore, we have
\begin{align}\label{ineq10}
\mathbb{E}\left\{ \left| J_{i,j}^{(a),T}(\theta) - \mathbb{E}^{\ast}\{J_{i,j}^{(a),T}(\theta)\} \right|^{2}\right\} \leq \mathbb{E}^{\ast}\{ \|DJ_{i,j}^{(a),T}(\theta)\|_{\mathcal{H}}^{2} \}
\end{align}
by the Poincar\'{e} inequality. The next lemma shows the law of large numbers for $J_{i,j}^{(a),T}(\theta)$.

\begin{lemma}\label{lemma10}
Suppose that $H\in(1/4,1/2)$ and $i,j=1,\ldots,m_{0}(\theta)$. Then we have
\begin{align}\label{limit4}
\mathbb{E}^{\ast}\{\| DJ_{i,j}^{(a),T}(\theta) \|_{\mathcal{H}}^{2}\} \to 0
\end{align}
as $n\to\infty$.
\end{lemma}

The proof of Lemma \ref{lemma10} is given in Section 4.2.
On the other hand, $\lim_{T\to\infty} \mathbb{E}^{\ast}J_{i,j}^{(a),T}(\theta) = I_{i,j}(\theta)$ by (\ref{eq22}). Therefore the limit (\ref{limit6}) holds in this case.

\noindent \textbf{Case (b)}. We assume $i=1,\ldots,m_{0}(\theta)$ and $j=m_{0}(\theta)+1,\ldots,m$. The case where $j=1,\ldots,m_{0}(\theta)$ and $i=m_{0}(\theta)+1,\ldots,m$ can be dealt with similarly. In this case, $\kappa(i,\theta)=-1/2$, $\kappa(j,\theta)=H-1$, and
\[
I_{i,j}^{T}(\theta)(\bar{X}^{\theta}) = T^{H-3/2} \int_{0}^{T} (\gamma_{t}^{2,i}(\theta) + \gamma_{t}^{3,i}(\theta) ) (\gamma_{t}^{1,j}(\theta) + \gamma_{t}^{2,j}(\theta) + \gamma_{t}^{3,j}(\theta) )\,dt.
\]
Therefore we obtain  
\begin{align}\label{ineq11}
\begin{split}
& \mathbb{E}\left\{ \left|I_{i,j}^{T}(\theta)(\bar{X}^{\theta}) - T^{H-3/2} \int_{0}^{T} (\gamma_{t}^{2,i}(\theta) + \gamma_{t}^{3,i}(\theta) )\gamma_{t}^{1,j}(\theta)\,dt \right|\right\} \\
&\leq T^{H-3/2} \sum_{(p,q)\in\{2,3\}^{2}} \left( \mathbb{E}^{\ast}\int_{0}^{T} \gamma_{t}^{p,i}(\theta)^{2}\,dt\right)^{1/2}  \left(\mathbb{E}^{\ast}\int_{0}^{T} \gamma_{t}^{q,j}(\theta)^{2}\,dt\right)^{1/2}.
\end{split}
\end{align}
Let us set $J_{i,j}^{(b),T}(\theta) = T^{H-3/2} \int_{0}^{T} (\gamma_{t}^{2,i}(\theta) + \gamma_{t}^{3,i}(\theta) )\gamma_{t}^{1,j}(\theta)\,dt $. As in the case (a), by Lemma \ref{corollary1}, we have 
\[
\mathbb{E}\left\{ \left|I_{i,j}^{T}(\theta)(\bar{X}^{\theta}) - J_{i,j}^{(b),T}(\theta) \right|\right\} \to 0
\]
as $T\to\infty$. The Poincar\'{e} inequality gives
\begin{align}\label{ineq12}
\mathbb{E}\{ |J_{i,j}^{(b),T}(\theta) - \mathbb{E}^{\ast}\{J_{i,j}^{(b),T}(\theta)\}|^{2} \} \leq \mathbb{E}^{\ast}\{ \| DJ_{i,j}^{(b),T}(\theta) \|^{2}_{\mathcal{H}} \}.
\end{align}

As in the case (a), the following lemma holds.

\begin{lemma}\label{lemma11}
Suppose that $i=1,\ldots, m_{0}$ and $j= m_{0}(\theta)+1,\ldots,m$. Then we have
\begin{align}\label{limit5}
\mathbb{E}^{\ast}\{\| DJ_{i,j}^{(b),T}(\theta) \|_{\mathcal{H}}^{2}\} \to 0
\end{align}
as $n\to\infty$.
\end{lemma}

The proof of Lemma \ref{lemma11} is given in Section 4.2. The right hand side of (\ref{ineq12}) tends to $0$ as $T\to\infty$ thanks to Lemma \ref{lemma11}. Since $\mathbb{E}^{\ast}\{J_{i,j}^{(b),T}(\theta)\}=0$, we obtain (\ref{limit6}).

\noindent \textbf{Case (c)}. Let us consider the case where $i=m_{0}(\theta)+1,\ldots,m$ and $j=m_{0}(\theta)+1,\ldots,m$. In this case, we have $\kappa(i,\theta)=\kappa(j,\theta)=H-1$ and 
\begin{align*}
&I_{i,j}^{T}(\theta)(\bar{X}^{\theta}) \\
&= T^{2H-2} \int_{0}^{T} (\gamma_{t}^{1,i}(\theta) + \gamma_{t}^{2,i}(\theta) + \gamma_{t}^{3,i}(\theta) ) (\gamma_{t}^{1,j}(\theta) + \gamma_{t}^{2,j}(\theta) + \gamma_{t}^{3,j}(\theta) )\,dt.
\end{align*}
Note that $I_{i,j}(\theta)=T^{2H-2} \int_{0}^{T} \gamma_{t}^{1,i}(\theta) \gamma_{t}^{1,j}(\theta)\,dt$. Therefore we can directly obtain 
\begin{align}\label{ineq13}
\begin{split}
&\mathbb{E}^{\ast}\{ |I_{i,j}^{T}(\theta)(\bar{X}^{\theta}) - I_{i,j}(\theta) | \} \\
&\leq T^{2H-2} \sum_{\substack{(p,q)\in\{ 1,2,3 \}^{2} \\ (p,q)\neq(1,1)}} \left( \mathbb{E}^{\ast}\int_{0}^{T} \gamma_{t}^{p,i}(\theta)^{2} \,dt \right)^{1/2} \left(  \mathbb{E}^{\ast}\int_{0}^{T} \gamma_{t}^{q,j}(\theta)^{2} \,dt \right)^{1/2}
\end{split}
\end{align}
in this case. Lemma \ref{corollary1} implies the right hand side of (\ref{ineq13}) converges to $0$ as $T\to\infty$. This completes the proof. 
\end{proof}

Next we show that the negligible part is indeed negligible. We start with the following estimation for $R_{t}^{T}(\theta,u)$.
\begin{lemma}
Let $K$ be a compact subset of $\mathbb{R}^{m}$. Then there exists a positive constant $C=C(H,\sigma,\theta,K) >0$ depending only on $\sigma,H,\theta$ and $K$ such that 
\[
\mathbb{E}^{\ast}\left\{ \sup_{u\in K\cap(\varphi_{T}(\theta)^{-1}(\Theta-\theta))}|R_{t}^{T}(\theta,u)(X^{\theta})|^{2} \right\} \leq C T^{-2} t^{1-2H}.
\]
\end{lemma}
\begin{proof}
Since 
\begin{align*}
&| \partial_{\theta}\beta_{t}(\theta+\epsilon \varphi_{T}(\theta) u)(x) - \partial_{\theta}\beta_{t}(\theta)(x) | \\
&\leq \bar{d}_{H}^{-1}|\sigma|^{-1} t^{H-1/2}\int_{0}^{t}ds\ (t-s)^{-1/2-H}s^{1/2-H} \\
&\quad \times | \partial_{\theta}a(x_{s},\theta+\epsilon \varphi_{T}(\theta) u) - \partial_{\theta}a(x_{s},\theta) | \\
&\lesssim\epsilon T^{-1/2} |u| \int_{0}^{t}ds\ (t-s)^{-1/2-H}s^{1/2-H}(1+|x_{s}|^{p}),
\end{align*}
holds by Assumption \ref{assumption1}, we have
\begin{align*}
& |R_{t}^{T}(\theta,u)| \\
&\leq T^{-1/2}|u| \int_{0}^{1}d\epsilon\ | \partial_{\theta}\beta_{t}(\theta+\epsilon \varphi_{T}(\theta)u) - \partial_{\theta}\beta_{t}(\theta) | \\
&\lesssim T^{-1}|u|^{2} t^{H-1/2} \int_{0}^{t}ds\ (t-s)^{-1/2-H}s^{1/2-H}(1+|X^{\theta}_{s}|^{p}).
\end{align*}
Therefore, we obtain 
\begin{align*}
& \mathbb{E}^{\ast}\left\{ \sup_{u\in K\cap(\varphi_{T}(\theta)^{-1}(\Theta-\theta))}|R_{t}^{T}(\theta,u)(X^{\theta})|^{2} \right\} \\
&\leq C(H,\sigma,\theta,K) T^{-2} t^{2H-1} \int_{0}^{t}ds\int_{0}^{t}dv\ (t-s)^{-1/2-H} \\
& \quad \times (t-v)^{-1/2-H} s^{1/2-H}v^{1/2-H} \mathbb{E}^{\ast}\{ (1+|X^{\theta}_{s}|^{p})(1+|X^{\theta}_{v}|^{p}) \} \\
&\leq T^{-2} C(H,\sigma,\theta,K) t^{1-2H}.
\end{align*}
Note that we used (\ref{ineq7}) in the last inequality. 
\end{proof}
Let us denote the four terms in the negligible part by $ N^{T}_{i}(\theta,u) $ $(i=1,2,3,4)$ in order. 

\begin{proposition}\label{lemma7}
Let $K$ be a compact subset of $\mathbb{R}^{m}$. Then we have, for each $i=1,2,3,4$,
\begin{align}\label{limit3}
 | N_{i}^{T}(\theta,u_{T}) | \to^{p} 0
\end{align}
as $T\to\infty$ for any sequence $u_{T}\in K\cap(\varphi_{T}(\theta)^{-1}(\Theta-\theta))$.
\end{proposition}
\begin{proof}
For $i=1$, the limit (\ref{limit3}) follows from Proposition \ref{lemma8}. The term $N_{2}^{T}(\theta,u_{T})$ can be bounded in $L^{2}(\mathbb{P}^{\ast})$:
\begin{align*}
\nonumber \mathbb{E}^{\ast}\left\{ |N_{2}^{T}(\theta,u_{T})|^{2} \right\} &= \int_{0}^{T}dt\ \mathbb{E}^{\ast}\{ | R_{t}^{T}(\theta,u_{T}) |^{2} \} \\
&\leq C(H,\sigma,\theta,K) T^{-2H}.
\end{align*}
Therefore we have $ | N_{2}^{T}(\theta,u_{T}) | = o_{p}(1) $. For the term $N_{3}^{T}(\theta,u_{T})$, we have 
\begin{align}
| N_{3}^{T}(\theta,u_{T}) | \leq (u^{\star}I^{T}(\theta)u)^{1/2}\left( \int_{0}^{T}dt\ |R_{t}^{T}(\theta,u_{T})|^{2} \right)^{1/2}
\end{align}
by H\"{o}lder's inequality. Here $ I^{T}(\theta) $ denotes the matrix $ (I^{T}_{i,j}(\theta))_{i,j=1,\ldots,m} $ (see (\ref{eq20})). We have $ | N_{3}^{T}(\theta,u_{T}) | = o_{p}(1) $ since $ u^{\star}I^{T}(\theta)u = O_{p}(1) $ holds by Proposition \ref{lemma8}. Finally, $ | N_{4}^{T}(\theta,u_{T}) | = o_{p}(1) $ is obvious.
\end{proof}

By Propositions 4.2 and 4.7, we complete the proof of the second part of Theorem \ref{theorem2}.

\subsection{Proof of lemmas in Section 5.1}


First we prove Lemma \ref{corollary1}. Finding a good upper bound of $c_{i,j}(t)$ is a key ingredient for the proof of Lemma $\ref{corollary1}$.

\begin{lemma}\label{lemma6}
There exists a constant $C>0$ that is independent of $t$ such that
\[
|c_{i,j}(t)| \leq C t^{H-3/2}
\]
for all $t\geq1$ and $ i,j=1,\ldots,d $.
\end{lemma}

\begin{notation}
We denote the Ornstein-Uhlenbeck semigroup and its generater by $(P_{t})_{t\geq0}$ and $L$. The operator $L^{-1}$ denotes the pseudo-inverse of $L$. See Chapter 2 of \cite{nourdin2012normal} for detail.
\end{notation}

We start with the following general lemma.

\begin{lemma}\label{lemma17}
Let $F$ and $G$ be in $\mathbb{D}^{1,2}$ with $\mathbb{E}^{\ast}\{F\} = \mathbb{E}^{\ast}\{G\}=0$. Then the equality 
\begin{align}\label{eq24}
\nonumber \mathbb{E}^{\ast}\{FG\} &= \int_{0}^{\infty}d\tau\,e^{-\tau}\int_{\Omega^{\ast}}\mathbb{P}^{\ast}(d\omega)\int_{\Omega^{\ast}}\mathbb{P}^{\ast}(d\omega^{\prime})\\ 
&\quad\times \langle (DG)(\omega), (DF)(e^{-\tau}\omega+\sqrt{1-e^{-2\tau}}\omega^{\prime}) \rangle_{\mathcal{H}}
\end{align}
holds. 
\end{lemma}

\begin{proof}
By an integration by parts formula for Malliavin calculus (see Theorem 2.9.1 of \cite{nourdin2012normal}), we have 
\[
\mathbb{E}^{\ast}\{FG\} = \mathbb{E}^{\ast}\{\langle DG, -DL^{-1}F \rangle\}.
\]
By Proposition 2.9.3 of \cite{nourdin2012normal}, the random variable $-DL^{-1}F$ can be represented as
\begin{align*}
 -DL^{-1}F &= \int_{0}^{\infty}d\tau\,e^{-\tau}P_{\tau}(DF).
\end{align*}
Hence Mehler's formula yields (\ref{eq24}) for simple random variables $F\in\mathcal{S}$. By an approximation argument, we obtain (\ref{eq24}) for all $F\in\mathbb{D}^{1,2}$.
\end{proof}


\begin{proof}[Proof of Lemma \ref{lemma6}]
Using Lemma \ref{lemma17}, we have
\begin{align*}
c_{i,j}(t) 
&= \int_{0}^{\infty}d\tau\,e^{-\tau}\int_{\Omega^{\ast}}\mathbb{P}^{\ast}(d\omega)\int_{\Omega^{\ast}}\mathbb{P}^{\ast}(d\omega^{\prime})(\partial_{x}\partial_{\theta_{i}}a)(\bar{X}^{\theta}_{t}(\omega),\theta) \\&\quad\times(\partial_{x}\partial_{\theta_{j}}a)(\bar{X}^{\theta}_{0}(e^{-\tau}\omega+\sqrt{1-e^{-2\tau}}\omega^{\prime}),\theta) \\&\quad\times \langle (D\bar{X}^{\theta}_{t})(\omega), (D\bar{X}^{\theta}_{0})(e^{-\tau}\omega+\sqrt{1-e^{-2\tau}}\omega^{\prime}) \rangle_{\mathcal{H}} \\
 &= \int_{0}^{\infty}d\tau\,e^{-\tau}\int_{\Omega^{\ast}}\mathbb{P}^{\ast}(d\omega)\int_{\Omega^{\ast}}\mathbb{P}^{\ast}(d\omega^{\prime})(\partial_{x}\partial_{\theta_{i}}a)(\bar{X}^{\theta}_{t}(\omega),\theta) \\&\quad\times(\partial_{x}\partial_{\theta_{j}}a)(\bar{X}^{\theta}_{0}(e^{-\tau}\omega+\sqrt{1-e^{-2\tau}}\omega^{\prime}),\theta) \\&\quad\times \langle \Psi(t,\cdot)(\omega), \Psi(0,\cdot)(e^{-\tau}\omega+\sqrt{1-e^{-2\tau}}\omega^{\prime}) \rangle_{\mathcal{H}}.
\end{align*}
Recall that the random function $\Psi(t,\cdot)$ is defined by
\[
 \Psi(t,r) = \sigma \mathbf{1}_{(-\infty,t]}(r) e^{\int_{r}^{t}(\partial_{x}a)(\bar{X}^{\theta}_{u},\theta)\,du}.
\]
The random function $(\mathbf{D}^{1/2-H}_{-}\Psi(t,\cdot))_{s}$ is decomposed as 
\begin{align*}
(\mathbf{D}^{1/2-H}_{-}\Psi(t,\cdot))_{s} &= \int_{0}^{\infty}d\xi\,\xi^{H-3/2}(\Psi(t,s)-\Psi(t,s+\xi)) \\
&= \int_{0}^{t-s}d\xi\,\xi^{H-3/2}(\Psi(t,s)-\Psi(t,s+\xi)) \\
&\quad + \int_{t-s}^{\infty}d\xi\,\xi^{H-3/2}\Psi(t,s) \\
&= \mathbf{1}_{(t-1,t]}(s)\int_{0}^{t-s}d¥\xi\,\xi^{H-3/2} (\Psi(t,s)-\Psi(t,s+\xi)) \\
&\quad +\mathbf{1}_{(-\infty,t-1]}(s)\int_{0}^{t-s}d¥\xi\,\xi^{H-3/2} (\Psi(t,s)-\Psi(t,s+\xi))\\
&\quad +(1/2-H)^{-1}(t-s)^{1/2-H}\Psi(t,s) \\
&= \mathbf{1}_{(t-1,t]}(s)\int_{0}^{t-s}d¥\xi\,\xi^{H-3/2} (\Psi(t,s)-\Psi(t,s+\xi)) \\
&\quad +\mathbf{1}_{(-\infty,t-1]}(s)\int_{0}^{1}d¥\xi\,\xi^{H-3/2} (\Psi(t,s)-\Psi(t,s+\xi))\\
&\quad +\mathbf{1}_{(-\infty,t-1]}(s)\int_{1}^{t-s}d¥\xi\,\xi^{H-3/2} (\Psi(t,s)-\Psi(t,s+\xi))\\
&\quad +(1/2-H)^{-1}(t-s)^{1/2-H}\Psi(t,s).
\end{align*}
Hence if $s\leq0$ and $t\geq1$, then
\begin{align*}
|(\mathbf{D}^{1/2-H}_{-}\Psi(t,\cdot))_{s}| &\leq \mathbf{1}_{(-\infty,t-1]}(s)\int_{0}^{1}d\xi\,\xi^{H-3/2}(\Psi(t,s+\xi)-\Psi(t,s))\\
&\quad + \mathbf{1}_{(-\infty,t-1]}(s)\int_{1}^{t-s}d\xi\,\xi^{H-3/2}\Psi(t,s+\xi)\\
&\quad + (1/2-H)^{-1}(t-s)^{H-1/2}\Psi(t,s) \\
& \lesssim (e^{-\alpha t}\mathbf{1}_{(-\infty,t-1]}(s)+(t-s)^{H-3/2}\mathbf{1}_{(-\infty,t-1]}(s) \\
&\quad +(t-s)^{H-1/2}e^{-\alpha(t-s)}\mathbf{1}_{(-\infty,t]}(s))
\end{align*} 
holds. Therefore we have, for $t\geq1$,
\begin{align*}
&|\langle\Psi(t,\cdot)(\omega),\Psi(0,\cdot)(e^{-\tau}\omega+\sqrt{1-e^{-2\tau}}\omega^{\prime})\rangle_{\mathcal{H}}| \\
&\leq \int_{\mathbb{R}}|(\mathbf{D}^{1/2-H}_{-}\Psi(t,\cdot)(\omega))_{s}|| (\mathbf{D}^{1/2-H}_{-}\Psi(0,\cdot)(e^{-\tau}\omega+\sqrt{1-e^{-2\tau}}\omega^{\prime}))_{s} |\,ds \\
&\lesssim \biggl(e^{-\alpha t}\int_{-\infty}^{0}| \mathbf{D}^{1/2-H}_{-}\Psi(0,\cdot)(e^{-\tau}\omega+\sqrt{1-e^{-2\tau}}\omega^{\prime})_{s} |\,ds \\
&\quad + \int_{-\infty}^{0}| \mathbf{D}^{1/2-H}_{-}\Psi(0,\cdot)(e^{-\tau}\omega+\sqrt{1-e^{-2\tau}}\omega^{\prime})_{s} |(t-s)^{H-3/2}\,ds\\
&\quad +e^{-\alpha t}\int_{-\infty}^{0}| \mathbf{D}^{1/2-H}_{-}\Psi(0,\cdot)(e^{-\tau}\omega+\sqrt{1-e^{-2\tau}}\omega^{\prime})_{s} |(t-s)^{H-1/2}e^{\alpha s}\,ds \biggr)\\
&\lesssim \left(\int_{-\infty}^{0}| \mathbf{D}^{1/2-H}_{-}\Psi(0,\cdot)(e^{-\tau}\omega+\sqrt{1-e^{-2\tau}}\omega^{\prime})_{s} |\,ds\right) t^{H-3/2}\\
&\lesssim t^{H-3/2}.
\end{align*}
Note that we used the inequality (\ref{ineq22}) in the last line. This completes the proof.
\end{proof}

\begin{proof}[Proof of Lemma \ref{corollary1}]
\noindent \textit{(1)} We have 
\begin{align*}
&\int_{0}^{\infty}dr\int_{0}^{\infty}du\ r^{-H-1/2}u^{-H-1/2}|c_{i,j}(|r-u|)| \\
&= 2 \int_{0}^{\infty}du \int_{0}^{u}dr \ r^{-H-1/2}u^{-H-1/2} |c_{i,j}(u-r)| \\
&= 2 \int_{0}^{\infty}du\int_{0}^{u}dr\ (u-r)^{-H-1/2}u^{-H-1/2} |c_{i,j}(r)| \\
&= 2 \int_{0}^{\infty}dr\ |c_{i,j}(r)|\int_{r}^{\infty}du\ (u-r)^{-H-1/2}u^{-H-1/2} \\
&= 2 \int_{0}^{\infty}dr\ |c_{i,j}(r)|r^{-2H}\int_{1}^{\infty}du\ (u-1)^{-H-1/2}u^{-H-1/2}.
\end{align*}
The integral $ \int_{0}^{\infty}dr\ |c_{i,j}(r)|r^{-2H} $ converges by Lemma \ref{lemma6}. 

\noindent \textit{(2)} By the change of variable, we obtain
\begin{align*}
&\mathbb{E}^{\ast}\left\{T^{-1} \int_{0}^{T}\gamma_{t}^{2,i}(\theta)\gamma_{t}^{2,j}(\theta)\,dt \right\} \\
&= \sigma^{-2}\bar{d}_{H}^{-2}T^{-1}\int_{0}^{T}dt\int_{0}^{t}dr\int_{0}^{t}du\ r^{-H-1/2}u^{-H-1/2}c_{i,j}(|r-u|) \\
&= \sigma^{-2}\bar{d}_{H}^{-2}\int_{0}^{1}dt\int_{0}^{tT}dr\int_{0}^{tT}du\ r^{-H-1/2}u^{-H-1/2}c_{i,j}(|r-u|).
\end{align*}
By (\ref{ineq14}) and the dominated convergence theorem, we obtain (\ref{eq22}).

\noindent \textit{(3)} As in the proof of (\ref{eq22}), we have
\begin{align*}
&\mathbb{E}^{\ast}\left\{T^{-1} \int_{0}^{T}\gamma_{t}^{3,i}(\theta)^{2}\,dt \right\} \\
&= \sigma^{-2}\bar{d}_{H}^{-2}\int_{0}^{1}dt\int_{0}^{tT}dr\int_{0}^{tT}du\ r^{-H-1/2}u^{-H-1/2}\left\{1-\left(1-\frac{r}{tT}\right)^{1/2-H}\right\}\\
&\quad \times\left\{1-\left(1-\frac{r}{tT}\right)^{1/2-H}\right\} c_{i,j}(|r-u|).
\end{align*}
Again by (\ref{ineq14}) and the dominated convergence theorem, (\ref{eq23}) follows.
\end{proof}

Now we turn to prove Lemmas \ref{lemma10} and \ref{lemma11}. The proofs are by straightforward calculation.

\begin{proof}[Proof of Lemma \ref{lemma10}]
By a straightforward calculation, we obtain
\begin{align*}
&\| DJ_{i,j}^{(a),T}(\theta) \|_{\mathcal{H}}^{2} \\
&= \sigma^{-4}\bar{d}_{H}^{-4}T^{-2}\int_{0}^{T}dt_{1}\int_{0}^{T}dt_{2}\int_{0}^{t_{1}}dr_{1}\int_{0}^{t_{1}}du_{1}\int_{0}^{t_{2}}dr_{2}\int_{0}^{t_{2}}du_{2}\Bigl\{ \\
&\quad (\partial_{x}\partial_{\theta_{i}}a)(\bar{X}^{\theta}_{r_{1}},\theta)(\partial_{\theta_{j}}a)(\bar{X}^{\theta}_{u_{1}},\theta)(\partial_{x}\partial_{\theta_{i}}a)(\bar{X}^{\theta}_{r_{2}},\theta)(\partial_{\theta_{j}}a)(\bar{X}^{\theta}_{u_{2}},\theta) \\ 
&\quad \times\langle D\bar{X}^{\theta}_{r_{1}},D\bar{X}^{\theta}_{r_{2}} \rangle_{\mathcal{H}} \\
&\quad + (\partial_{x}\partial_{\theta_{i}}a)(\bar{X}^{\theta}_{r_{1}},\theta)(\partial_{\theta_{j}}a)(\bar{X}^{\theta}_{u_{1}},\theta)(\partial_{x}\partial_{\theta_{j}}a)(\bar{X}^{\theta}_{u_{2}},\theta)(\partial_{\theta_{i}}a)(\bar{X}^{\theta}_{r_{2}},\theta) \\
&\quad \times\langle D\bar{X}^{\theta}_{r_{1}},D\bar{X}^{\theta}_{u_{2}} \rangle_{\mathcal{H}} \\
&\quad + (\partial_{x}\partial_{\theta_{j}}a)(\bar{X}^{\theta}_{u_{1}},\theta)(\partial_{\theta_{i}}a)(\bar{X}^{\theta}_{r_{1}},\theta)(\partial_{x}\partial_{\theta_{i}}a)(\bar{X}^{\theta}_{r_{2}},\theta)(\partial_{\theta_{j}}a)(\bar{X}^{\theta}_{u_{2}},\theta)\\
&\quad \times\langle D\bar{X}^{\theta}_{u_{1}},D\bar{X}^{\theta}_{r_{2}} \rangle_{\mathcal{H}} \\
&\quad + (\partial_{x}\partial_{\theta_{j}}a)(\bar{X}^{\theta}_{u_{1}},\theta)(\partial_{\theta_{i}}a)(\bar{X}^{\theta}_{r_{1}},\theta)(\partial_{x}\partial_{\theta_{j}}a)(\bar{X}^{\theta}_{u_{2}},\theta)(\partial_{\theta_{i}}a)(\bar{X}^{\theta}_{r_{2}},\theta)\\
&\quad \times\langle D\bar{X}^{\theta}_{u_{1}},D\bar{X}^{\theta}_{u_{2}} \rangle_{\mathcal{H}} \Bigr\}\\
&\quad \times (t_{1}-r_{1})^{-H-1/2}(t_{1}-u_{1})^{-H-1/2} (t_{2}-r_{2})^{-H-1/2} (t_{2}-u_{2})^{-H-1/2}.
\end{align*}
Let us set $\tilde{c}(t)=\mathbb{E}^{\ast}\{|\langle D\bar{X}^{\theta}_{t}, D\bar{X}^{\theta}_{0} \rangle_{\mathcal{H}}|^{2}\}^{1/2}$. As in the proof of Lemma \ref{lemma6}, there exists a positive constant $C$ that is independent of $t$ such that the inequality
\[
\tilde{c}(t) \leq Ct^{H-3/2}
\]
holds for each $t\geq1$. Hence it holds that
\begin{align*}
&\mathbb{E}^{\ast}\{\| DJ_{i,j}^{(a),T}(\theta) \|_{\mathcal{H}}^{2}\} \\
&\lesssim T^{-2}\int_{0}^{T}dt_{1}\int_{0}^{T}dt_{2}\int_{0}^{t_{1}}dr_{1}\int_{0}^{t_{1}}du_{1}\int_{0}^{t_{2}}dr_{2}\int_{0}^{t_{2}}du_{2} \bigl\{ \\
 &\quad \tilde{c}(|r_{1}-r_{2}|) + \tilde{c}(|r_{1}-u_{2}|) + \tilde{c}(|u_{1}-r_{2}|) + \tilde{c}(|u_{1}-u_{2}|) \bigr\} \\
 &\quad \times (t_{1}-r_{1})^{-H-1/2}(t_{1}-u_{1})^{-H-1/2} (t_{2}-r_{2})^{-H-1/2} (t_{2}-u_{2})^{-H-1/2} \\
 &= 4T^{-2}\int_{0}^{T}dt_{1}\int_{0}^{T}dt_{2}\int_{0}^{t_{1}}dr_{1}\int_{0}^{t_{1}}du_{1}\int_{0}^{t_{2}}dr_{2}\int_{0}^{t_{2}}du_{2}\ \tilde{c}(|r_{1}-r_{2}|) \\
 &\quad \times (t_{1}-r_{1})^{-H-1/2}(t_{1}-u_{1})^{-H-1/2} (t_{2}-r_{2})^{-H-1/2} (t_{2}-u_{2})^{-H-1/2} \\
 &= 4(1/2-H)^{-2}T^{-2}\int_{0}^{T}dt_{1}\int_{0}^{T}dt_{2}\ t_{1}^{1/2-H}t_{2}^{1/2-H} \\
 &\quad \times \int_{0}^{t_{1}}dr_{1}\int_{0}^{t_{2}}dr_{2} (t_{1}-r_{1})^{-H-1/2} (t_{2}-r_{2})^{-H-1/2}\tilde{c}(|r_{1}-r_{2}|) \\
 &= 4(1/2-H)^{-2}T^{-2} \int_{0}^{T}dr_{1}\int_{0}^{T}dr_{2}\ \tilde{c}(|r_{1}-r_{2}|)\\
 &\quad \times \int_{r_{1}}^{T}dt_{1}\ (t_{1}-r_{1})^{-H-1/2}t_{1}^{1/2-H} \int_{r_{2}}^{T}dt_{2}\ (t_{2}-r_{2})^{-H-1/2}t_{2}^{1/2-H} \\
 &\leq 4(1/2-H)^{-4}T^{-1-2H} \int_{0}^{T}dr_{1}\int_{0}^{T}dr_{2}\ \tilde{c}(|r_{1}-r_{2}|) \\
&\quad\times(T-r_{1})^{1/2-H}(T-r_{2})^{1/2-H} \\
 &= 8(1/2-H)^{-4}T^{-1-2H} \int_{0}^{T}dr_{2}\int_{0}^{r_{2}}dr_{1}\ \tilde{c}(r_{2}-r_{1}) r_{1}^{1/2-H}r_{2}^{1/2-H} \\
 &= 8(1/2-H)^{-4}T^{-1-2H} \Biggl( \int_{0}^{1}dr_{2}\int_{0}^{r_{2}}dr_{1}\ \tilde{c}(r_{2}-r_{1}) r_{1}^{1/2-H}r_{2}^{1/2-H} \\
 &\quad + \int_{1}^{T}dr_{2}\int_{0}^{r_{2}-1}dr_{1}\ \tilde{c}(r_{2}-r_{1})r_{1}^{1/2-H}r_{2}^{1/2-H}  \\ &\quad + \int_{1}^{T}dr_{2}\int_{r_{2}-1}^{r_{2}}dr_{1}\ \tilde{c}(r_{2}-r_{1})r_{1}^{1/2-H}r_{2}^{1/2-H} \Biggr).
\end{align*}
Let us denote the last three terms by $\mathtt{I}_{1}^{(a),T}(\theta)$, $\mathtt{I}_{2}^{(a),T}(\theta)$ and $\mathtt{I}_{3}^{(a),T}(\theta)$. It is clear that the term $\mathtt{I}_{1}^{(a),T}(\theta) = O(T^{-1-2H})$. The term $\mathtt{I}_{2}^{(a),T}(\theta)$ can be bounded as
\begin{align*}
\mathtt{I}_{2}^{(a),T}(\theta) &\lesssim T^{-1-2H} \int_{1}^{T}dr_{2}\ r_{2}^{1-2H}\int_{0}^{r_{2}-1}dr_{1}\ (r_{2}-r_{1})^{H-3/2} \\
&= O(T^{1-4H}).
\end{align*}
Finally, the term $\mathtt{J}_{3}^{(a),T}(\theta)$ is bounded as
\begin{align*}
\mathtt{I}_{3}^{(a),T}(\theta) &\lesssim T^{-1-2H} \int_{0}^{T}dr_{2}\ r_{2}^{1-2H} \\
&= O(T^{1-4H}).
\end{align*}
Since we assume $H\in(1/4,1/2)$, we obtain the limit (\ref{limit4}).
\end{proof}

\begin{proof}[Proof of Lemma \ref{lemma11}]
For notational simplicity, we set 
\[
C(\sigma,H,\theta) = \sigma^{-2}\bar{d}_{H}^{-2}B(-H+1/2,-H+3/2)\mathbb{E}^{\ast}\{(\partial_{\theta_{j}}a)(\bar{X}^{\theta}_{0},\theta)\}.
\]
As in the proof of Lemma \ref{lemma10}, we have 
\begin{align*}
&\| DJ_{i,j}^{(b),T}(\theta) \|_{\mathcal{H}}^{2} \\
&= \frac{C(\sigma,H,\theta)^{2}}{T^{3-2H}}\int_{0}^{T}dt_{1}\int_{0}^{T}dt_{2}\int_{0}^{t_{1}}dr_{1}\int_{0}^{t_{2}}dr_{2}\ (t_{1}-r_{1})^{-H-1/2}r_{1}^{1/2-H} \\
&\quad \times (t_{2}-r_{2})^{-H-1/2}r_{2}^{1/2-H}  (\partial_{x}\partial_{\theta_{i}}a)(\bar{X}_{r_{1}}^{\theta},\theta) (\partial_{x}\partial_{\theta_{i}}a)(\bar{X}_{r_{2}}^{\theta},\theta) \langle D\bar{X}^{\theta}_{r_{1}}, D\bar{X}^{\theta}_{r_{2}} \rangle_{\mathcal{H}}.
\end{align*}
Therefore we obtain 
\begin{align*}
&\mathbb{E}^{\ast}\{\| DJ_{i,j}^{(b),T}(\theta) \|_{\mathcal{H}}^{2}\} \\
&\leq \frac{C(\sigma,H,\theta)^{2}}{T^{3-2H}}\int_{0}^{T}dt_{1}\int_{0}^{T}dt_{2}\int_{0}^{t_{1}}dr_{1}\int_{0}^{t_{2}}dr_{2}\ (t_{1}-r_{1})^{-H-1/2}r_{1}^{1/2-H}\\
&\quad \times (t_{2}-r_{2})^{-H-1/2}r_{2}^{1/2-H}\tilde{c}(|r_{1}-r_{2}|) \\
&= \frac{(1/2-H)^{2}C(\sigma,H,\theta)^{2}}{T^{3-2H}}\int_{0}^{T}dr_{1}\int_{0}^{T}dr_{2}\ (T-r_{1})^{1/2-H}(T-r_{2})^{1/2-H}\\ 
&\quad \times r_{1}^{1/2-H}r_{2}^{1/2-H}\tilde{c}(|r_{1}-r_{2}|) \\
&\leq \frac{(1/2-H)^{2}C(\sigma,H,\theta)^{2}}{T^{2}}\int_{0}^{T}dr_{1}\int_{0}^{T}dr_{2}\ r_{1}^{1/2-H}r_{2}^{1/2-H}\tilde{c}(|r_{1}-r_{2}|)\\
&\leq o(1) + \frac{2(1/2-H)^{2}C(\sigma,H,\theta)^{2}}{T^{2}}\int_{1}^{T}dr_{2}\int_{1}^{r_{2}}dr_{1}\ (r_{2}-r_{1})^{1/2-H}\\ 
&\quad \times r_{2}^{1/2-H}r_{1}^{H-3/2}\\
&\leq o(1) + \frac{2(1/2-H)^{2}C(\sigma,H,\theta)^{2}}{T^{2}}(2-2H)^{-1}T^{2-2H}.
\end{align*}
This completes the proof.
\end{proof}


\bibliographystyle{abbrvnat}
\bibliography{foulan} 

\end{document}